\documentclass[11pt]{amsart}

\usepackage{graphicx}  
\usepackage{color}
\usepackage{enumerate}
\usepackage{thm-restate}
\usepackage{hyperref}
\usepackage{cleveref}
\usepackage{todonotes}
\usepackage{lscape}
\usepackage{multirow}
\usepackage{amsmath}
\usepackage{tikz}
\usetikzlibrary{decorations.markings}
\usetikzlibrary{decorations.pathreplacing,calligraphy}
\usepackage{enumitem}
\usepackage{MnSymbol}

\newlist{steps}{enumerate}{1}
\setlist[steps, 1]{label = Step \arabic*}
\newlist{primenumerate}{enumerate}{1}
\setlist[primenumerate,1]{label={(\arabic*$'$)}}

\newcommand{\set}[1]{ \left\{ #1 \right\} }

\newcommand{\aint}{\hat{\iota}}

\newcommand{\GG}{\mathcal{G}}
\newcommand{\extG}{\widehat{\GG}}

\newcommand{\as}{\alpha_*}
\newcommand{\bs}{\beta_*}
\newcommand{\gs}{\gamma_*}
\newcommand{\ds}{\delta_*}
\newcommand{\es}{\varepsilon_*}

\newcommand{\goodbasis}{slope basis}

\title{Genus two Goeritz equivalence in $S^3$}

\author{Brandy Doleshal}
\address{Brandy Doleshal
\newline Sam Houston State University
\newline Huntsville, TX
\newline USA}
\email{bdoleshal@shsu.edu}
\urladdr{}

\author{Matt Rathbun}
\address{Matt Rathbun
\newline California State University, Fullerton
\newline Fullerton, CA
\newline USA}
\email{mrathbun@fullerton.edu}
\urladdr{}
\newtheorem{thm}{Theorem}[section]    
\newtheorem{lem}[thm]{Lemma}          
\theoremstyle{definition}
\newtheorem{defn}[thm]{Definition}    
\newtheorem{rem}{Remark}
\newtheorem{prop}[thm]{Proposition}
\newtheorem{cor}[thm]{Corollary}

\newtheorem*{question}{Question}
\begin{document}

\begin{abstract}    
The Goeritz group of a genus $g$ Heegaard splitting of a 3-manifold is the group of isotopy classes of orientation-preserving automorphisms of the manifold that preserve the Heegaard splitting. In the context of the standard genus 2 Heegaard splitting of $S^3$, we introduce the concept of Goeritz equivalence of curves, present two algebraic obstructions to Goeritz equivalence of simple closed curves that are straightforward to compute, and provide families of examples demonstrating how these obstructions may be used.
\end{abstract}

\maketitle


\section{Introduction}
\label{section:Introduction}

The Goeritz group of a genus $g$ Heegaard splitting of a 3-manifold was defined by Goeritz \cite{GoeDAdBudVvG} as the group of isotopy classes of orientation-preserving automorphisms of the manifold preserving the Heegaard splitting (set-wise). Scharlemann \cite{SchaA3SPG2HS}, Akbas \cite{AkbPA3SPG2HS}, Cho \cite{ChoH3SPHSG2}, and Powell \cite{PowHS3lHsi} have done work to define generators for the Goeritz group in order to solve a foundational question of whether the Goeritz group for the genus $g$ splitting of $S^3$ is finitely generated. 

We focus on the Goeritz group in the setting of a genus $2$ Heegaard splitting, $F$, of $S^3$. In \cite{BergeUnpub}, Berge finds distinct primitive/primitive knots that have the same lens space surgery at the same surface slope. It is unknown whether there exist pairs of distinct primitive/Seifert knots that have the same Seifert fibered space surgery at the same surface slope. However in \cite{GunKDPPPSR}, \cite{DolFPSTTK} and \cite{AmoDolRatACPTTK}, the authors, along with Evan Amoranto, find infinite families of primitive/Seifert knots that have distinct positions on the genus 2 surface. Infinite families with such distinct positions were also found by Eudave-Mu\~{n}oz, Miyazaki, and Motegi \cite{EudMunMiyMotSFSDPSP}. While this previous work was in the service of understanding Dehn surgery, we zoom out from the specialized primitive/primitive and primitive/Seifert constructions to consider any non-separating curve in the standard genus 2 Heegaard surface, $F$, for $S^3$. However, in light of this context, we will also allow for automorphisms that exchange the two handlebodies, considering the \emph{extended Goeritz group}.

We seek to determine when there exist elements of the Goeritz group, or the extended Goeritz group, that send one given curve in the standard genus 2 Heegaard splitting of $S^3$ to another. If such an element from the Goeritz group exists, we call the two curves \emph{Goeritz equivalent}.

Observe that two curves that are Goeritz equivalent represent the same knot, but that the same knot may be represented by multiple Goeritz equivalence classes.  In particular, the authors' previous work with primitive/Seifert knots shows that primitive/Seifert knots can belong to two distinct Goeritz equivalence classes. Understanding Goeritz equivalence and the Goeritz equivalence classes can lead to understanding questions about distinct representatives of primitive/Seifert knots and much more. The goal of this work is to move toward a better understanding of Goeritz equivalence. 

To that end, we establish a simple quantity called the \emph{split product}, that computes the surface slope of a curve.

\begin{restatable}{thm}{SplitProductSurfaceSlope} \label{thm:SplitProductSurfaceSlope} Suppose $K \subset F$ is a non-separating simple closed curve. The split product of $[K]  \in H_1(F)$ with respect to any \goodbasis \ is equal to the surface slope of $K$ with respect to $F$.
\end{restatable}

We then provide two algebraic obstructions to Goeritz equivalence.

The first is homological and arithmetic. We will say that two curves are \emph{homologically Goeritz equivalent} if there is an element of the Goeritz group whose induced action on homology takes the homology class of the first curve to that of the second. Equivalently, $K$ is homologically Goeritz equivalent to $K'$ if there is a Goeritz group element that takes $K$ to a curve homologous to $K'$.

\begin{restatable}{thm}{TheoremHomology}\label{thm:Homology}
Suppose $K$ and $K'$ are non-separating nontrivial simple closed curves in the standard genus two Heegaard surface, $F$, for $S^3$ with non-zero surface slope, and suppose $[K] = (a, x, b, y)^T$ and $[K'] = (a', x', b', y')^T$ with respect to any slope basis. Then $K$ and $K'$ are homologically Goeritz equivalent if and only if:
 \begin{enumerate} 
 \item \label{cond:firstgcd} $gcd(a,x) = gcd(a',x')$, 
 \item \label{cond:secondgcd} $gcd(b,y) = gcd(b',y')$, and 
 \item \label{cond:slopes} $ab+xy = a'b'+x'y'$,
\end{enumerate}
and either
\begin{enumerate}[start=4]
\item \label{cond:topleft} $a'b + xy'$,
\item \label{cond:topright} $a'y - ay'$, 
\item \label{cond:bottomleft} $bx' - b'x ,$ and 
\item \label{cond:bottomright} $x'y + ab'$ 
\end{enumerate}
are all multiples of $ab + xy$, or
\begin{primenumerate}[start=4]
\item \label{cond:topleftprime} $a'b - xy'$, 
\item \label{cond:toprightprime} $a'y + ay'$, 
\item \label{cond:bottomleftprime} $bx' + b'x ,$ and 
\item \label{cond:bottomrightprime} $x'y - ab'$
\end{primenumerate}
are all multiples of $ab + xy$.
\end{restatable}

To determine Goeritz equivalence of two simple closed curves, then, it suffices to consider equivalence between homologous curves. The next obstruction is a matter of homotopy, relative to the two handlebodies bounded by $F$, $H$ and $H'$, separately.

\begin{restatable}{thm}{TheoremHomotopy}\label{thm:Homotopy}
Suppose $K$ and $K'$ are homologous simple closed curves with non-zero surface slopes. If $K$ is not freely homotopic to $K'$ in $H$ or if $K$ is not freely homotopic to $K'$ in $H'$, then $K$ and $K'$ are not Goeritz equivalent.
\end{restatable}

In Sections \ref{subsection:ExtendedGoeritzGroup} and \ref{subsection:ActionHomology}, we provide some necessary background about the (extended) Goeritz group and how it acts on the homology vectors for $F$. In Section \ref{subsection:SplitProduct}, we define the the split product and prove Theorems \ref{thm:SplitProductSurfaceSlope} and \ref{thm:Homology}. In Section \ref{subsection:FundamentalGroups}, we prove Theorem \ref{thm:Homotopy}. In Section \ref{section:examples}, we demonstrate some examples of the obstruction methods in action.

\subsection{Acknowledgements}
The authors would like to thank Chris Lyons, Adam Glesser, and Dan Margalit for helpful conversations, and to give special thanks to Ken Baker, and the referee who helped improve this manuscript. The first author was partially supported by the Sam Houston State University Individual Scholarship Program. 

\section{The Genus 2 Extended Goeritz Group in $S^3$}
\label{section:GoeritzGroup}

\subsection{Definitions}
\label{subsection:ExtendedGoeritzGroup}

Akbas gives a presentation for the Goeritz group of the genus 2 Heegaard splittings in $S^3$ \cite{AkbPA3SPG2HS}. In \cite{SchaA3SPG2HS}, Scharlemann also provides a presentation for the Goeritz group of the genus 2 Heegaard splitting in $S^3$, showing that the group is generated by four elements, called $\alpha$, $\beta$, $\gamma$ and $\delta$. The elements $\alpha$ and $\gamma$ are rotations by $\pi$ about the horizontal and vertical axes of symmetry of $F$, i.e., the axes through the blue stars and green circles, respectively, shown in Figure \ref{fig:goeritz}. The element $\beta$ is a rotation of half of $F$ by $\pi$ about the axis through the blue stars, with the right handle of the surface rotating clockwise, as viewed from the right. The map $\beta^2$, then can be considered as a left-handed Dehn twist about the curve $C$, which we refer to as the \emph{belt curve}, shown in Figure \ref{fig:goeritz}. As mentioned in \cite{SchaA3SPG2HS}, $\delta$ can represent several different maps. In this paper, we will consider $\delta$ to be the operation that slides the top foot of the righthand handle of the handlebody over the other handle in the counterclockwise direction, with reference to Figure \ref{fig:goeritz}. This differs from \cite{AkbPA3SPG2HS} where $\delta$ is considered to be the rotation by $2\pi/3$ about a vertical axis through the handlebody, as pictured in Figure \ref{fig:deltarot}. We will refer to these two different conceptions of $\delta$ as $\delta_{slide}$ and $\delta_{rot}$, respectively. The relationship between $\delta_{slide}$ and $\delta_{rot}$ is given by

\[ \delta_{rot} = \beta^{-1} \gamma \delta_{slide}.\]

\begin{center}
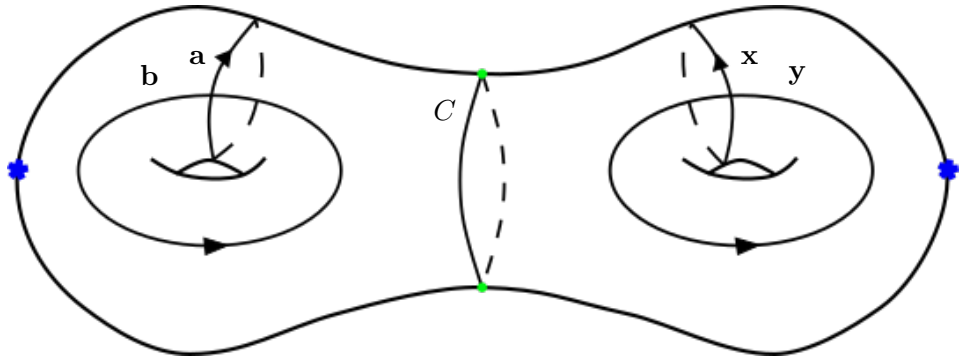
\begin{figure}
\begin{tikzpicture}
    \node[anchor=south west,inner sep=0] (image) at (0,0) {\includegraphics[width=\linewidth]{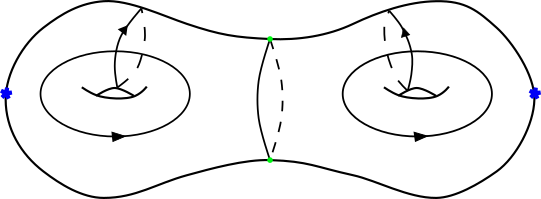}};
    \begin{scope}[
        x={(image.south east)},
        y={(image.north west)}
    ]
        \node [font=\bfseries] at (0.20,0.85) {$\textbf{a}$};
        \node [font=\bfseries] at (0.15,0.8) {$\textbf{b}$};
        \node [font=\bfseries] at (0.78,0.85) {$\textbf{x}$};
        \node [font=\bfseries] at (0.83,0.8) {$\textbf{y}$};
        \node [font=\bfseries] at (0.46,0.7) {$C$};
    \end{scope}
\end{tikzpicture}
\caption{The genus 2 surface with markings to indicate the Goeritz group elements.}\label{fig:goeritz}
\end{figure}
\end{center}

\begin{figure}
\begin{center}
\begin{tikzpicture}
    \node[anchor=south west,inner sep=0] (image) at (0,0) {\includegraphics[scale = .9]{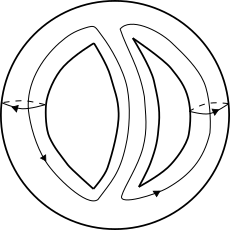}};
    \begin{scope}[
        x={(image.south east)},
        y={(image.north west)}
    ]
        \node [font=\bfseries] at (0.07,0.47) {$\textbf{a}$};
        \node [font=\bfseries] at (0.11,0.4) {$\textbf{b}$};
        \node [font=\bfseries] at (0.96,0.47) {$\textbf{x}$};
        \node [font=\bfseries] at (0.94,0.4) {$\textbf{y}$};
    \end{scope}
\end{tikzpicture}
\caption{The genus 2 surface imagined with a 3-fold rotational symmetry.}\label{fig:deltarot}
\end{center}
\end{figure}
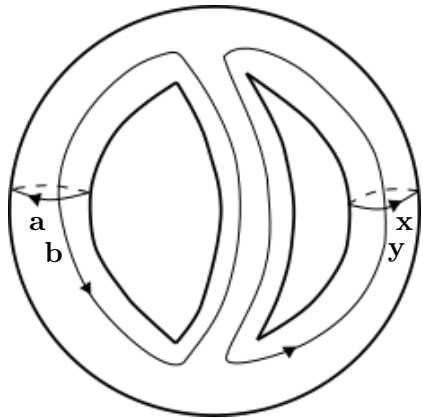

\begin{figure}
\begin{center}
\begin{tikzpicture}
    \node[anchor=south west,inner sep=0] (image) at (0,0) {\includegraphics[scale = .9]{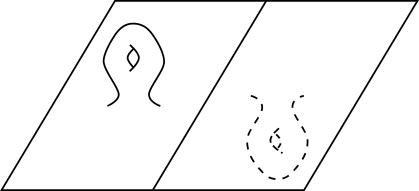}};
    \begin{scope}[
        x={(image.south east)},
        y={(image.north west)}
    ]
               \node [font=\bfseries] at (0.5,0.7) {$C$};
    \end{scope}
\end{tikzpicture}
\caption{The genus 2 surface with a point at infinity.}\label{fig:epsilon}
\end{center}
\end{figure}

Let $\GG$ denote the genus 2 Goeritz group for $S^3$. Then, adapting the amalgamated free product exhibited by Akbas (using $\delta_{rot}$) in \cite{AkbPA3SPG2HS}, and using $\delta = \delta_{slide}$ as a generator, we provide an explicit presentation:
\[ \GG = \left\langle \alpha, \beta, \gamma, \delta \, \, | \, \, \alpha^2, \, \, \gamma^2, \, \, \left[ \alpha, \beta\right],    
 \, \, \left[\alpha, \gamma\right], \, \, \left[\alpha, \delta\right], \, \, \alpha \left[ \gamma, \beta \right], \, \, (\beta^{-1} \gamma \delta)^3, \, \, (\beta^{-1} \delta)^2 \right\rangle,\]
 where $[a, b]$ is the commutator of $a$ and $b$, $aba^{-1}b^{-1}$.
 
In addition to these elements, the authors of this paper, along with E. Amoranto \cite{AmoDolRatACPTTK}, introduce an element $\varepsilon$, to extend the Goeritz group to another group that fixes the set $\{F, H, H'\}$, but does not fix $H$ and $H'$ individually, as the elements of the Goeritz group do. This element can be seen as a rotation about the curve $C$, shown in Figure \ref{fig:epsilon}. We call this group the \textit{extended Goeritz group}.

The extended Goeritz group, denoted $\extG$, is the extension of $\left\langle \varepsilon \, \, | \, \, \varepsilon^2 \right\rangle$ by $\GG$,
\[ \extG = \GG \rtimes \left\langle \varepsilon \, \, | \, \, \varepsilon^2 \right\rangle.\]

One can verify the relations,
\[ \varepsilon \alpha \varepsilon = \alpha, \qquad \varepsilon \beta \varepsilon = \alpha \beta, \qquad \varepsilon \gamma \varepsilon = \gamma, \qquad \varepsilon \delta \varepsilon = \delta^{-1}, \]

so that a presentation for $\extG$ is


\begin{eqnarray*} \extG = \left\langle \alpha, \beta, \gamma, \delta, \varepsilon \, \, | \, \, \alpha^2, \, \, \gamma^2, \, \, \varepsilon^2, \, \,  \left[ \alpha, \beta\right], \, \, \left[\alpha, \gamma\right], \, \, \left[\alpha, \delta\right], \, \, \alpha \left[ \gamma, \beta \right], \right. \qquad \qquad \qquad \\  \left. \left[ \varepsilon, \alpha \right], \, \, \alpha \left[ \varepsilon, \beta\right], \, \, \left[ \varepsilon, \gamma\right], \, \,  \left(\varepsilon \delta \right)^2, \, \, (\beta^{-1} \gamma \delta)^3, \, \, (\beta^{-1} \delta)^2 \right\rangle.  \end{eqnarray*}

\subsection{Action on homology}
\label{subsection:ActionHomology}
Each element of the (extended) Goeritz group induces an action on $H_1(F; \mathbb{Z}) \cong \mathbb{Z}^4$. Thus, we have homomorphisms
\[ * : \GG \to GL(4, \mathbb{Z}),\]
and 
\[ * : \extG \to GL(4, \mathbb{Z}).\]

We denote by $\alpha_*$, $\beta_*$, $\gamma_*$, $\delta_*$ and $\varepsilon_*$ the induced maps on $H_1(F)$ by $\alpha$, $\beta$, $\gamma$, $\delta$ and $\varepsilon$, respectively. The curves in Figure \ref{fig:goeritz} form a basis for $H_1(F)$. With a slight abuse of notation, we will refer both to these specific curves and their homology vectors with the same notation, and call $\set{\textbf{a}, \textbf{x}, \textbf{b}, \textbf{y} }$ the \emph{standard basis} for $H_1(F)$. The induced maps on $H_1(F)$ act on these generators of $H_1(F)$ in the following ways:

\[
\begin{array}{lllll}
\alpha_* : & \textbf{a} \mapsto -\textbf{a},  & \textbf{x} \mapsto -\textbf{x}, & \textbf{b} \mapsto -\textbf{b}, & \textbf{y} \mapsto -\textbf{y} \\
\beta_*: & \textbf{a} \mapsto \textbf{a}, & \textbf{x} \mapsto -\textbf{x}, & \textbf{b} \mapsto \textbf{b}, & \textbf{y} \mapsto -\textbf{y} \\
\gamma_*: &  \textbf{a} \mapsto -\textbf{x}, & \textbf{x} \mapsto -\textbf{a}, & \textbf{b} \mapsto -\textbf{y}, & \textbf{y} \mapsto -\textbf{b} \\
\delta_*: & \textbf{a} \mapsto \textbf{a}+\textbf{x}, & \textbf{x} \mapsto \textbf{x}, & \textbf{b} \mapsto \textbf{b}, & \textbf{y} \mapsto \textbf{y}-\textbf{b} \\
\varepsilon_*: &  \textbf{a} \mapsto \textbf{y}, & \textbf{x} \mapsto \textbf{b}, & \textbf{b} \mapsto \textbf{x}, & \textbf{y} \mapsto \textbf{a},
\end{array}
\]
so that the images of the generators of $\widehat{\GG}$ in $GL(4, \mathbb{Z})$ are:

\begin{gather*}
\alpha_* = \begin{pmatrix} -1 & 0 & 0 & 0 \\ 0 & -1 & 0 & 0 \\ 0 & 0 & -1 & 0 \\ 0 & 0 & 0 & -1\end{pmatrix}, \qquad \beta_*  = \begin{pmatrix} 1 & 0 & 0 & 0 \\ 0 & -1 & 0 & 0 \\ 0 & 0 & 1 & 0 \\ 0 & 0 & 0 & -1\end{pmatrix}, \qquad \\ \gamma_*  = \begin{pmatrix} 0 & -1 & 0 & 0 \\ -1 & 0 & 0 & 0 \\ 0 & 0 & 0 & -1 \\ 0 & 0 & -1 & 0\end{pmatrix}, \ 
\delta_*   =  \begin{pmatrix} 1 & 0 & 0 & 0 \\ 1 & 1 & 0 & 0 \\ 0 & 0 & 1 & -1 \\ 0 & 0 & 0 & 1\end{pmatrix}, \mbox{ and } \varepsilon_* = \begin{pmatrix} 0 & 0 & 0 & 1 \\ 0 & 0 & 1 & 0 \\ 0 & 1 & 0 & 0 \\ 1 & 0 & 0 & 0\end{pmatrix}.
\end{gather*}

Thus, considering an element of $H_1(F)$ as a vector of the form $(a,x,b,y)^T$, the induced maps have the following effect:

\[
\begin{array}{lcl}
\alpha_* ((a, x, b, y)^T) &=& (-a, -x, -b, -y)^T \\
\beta_* ((a, x, b, y)^T) &=& (a, -x, b, -y)^T \\
\gamma_* ((a, x, b, y)^T) &=& (-x, -a, -y, -b)^T \\
\delta_* ((a, x, b, y)^T) & = & (a, a+x, b-y, y)^T \\
\varepsilon_* ((a, x, b, y)^T) &=& (y, b, x, a)^T. \\
\end{array}
\]

Observe, then, that the image of $*: \GG \to GL(4, \mathbb{Z})$ is contained in the subgroup
\[ \mathcal{D} = \set{ \begin{pmatrix} C & 0 \\ 0 & (C^{-1})^T \end{pmatrix} \, \, | \, \, C \in GL(2, \mathbb{Z}) } \cong GL(2, \mathbb{Z}), \] 
where the isomorphism is the projection, $p$, to the first block $2 \times 2$ matrix. We will refer to such block diagonal matrices with the two blocks being inverse transposes of each other as being in \emph{Goeritz-form}.

One may then easily identify $\alpha_*$, $\beta_*$, $\gamma_*$, and $\delta_*$ with elements of $GL(2, \mathbb{Z})$. In fact, a presentation for $GL(2, \mathbb{Z})$ is given (see \cite{CoxMosGRDG}) by
\[  \left\langle R_1, R_2, R_3 \, \, | \, \, (R_1)^2, \, \, (R_2)^2, \, \, (R_3)^2, \, \, (R_1R_2)^3(R_1R_3)^{-2}, \, \, (R_1R_2)^6, \, \, (R_1R_3)^4 \right\rangle,\] 
where $R_1 = \begin{pmatrix} 0 & 1 \\ 1 & 0 \end{pmatrix}$, $R_2 = \begin{pmatrix} -1 & 0 \\ 1 & 1 \end{pmatrix}$, and $R_3 = \begin{pmatrix} -1 & 0 \\ 0 & 1 \end{pmatrix}$, and it can be shown directly that $p \circ *$ is surjective, since $R_1 = p(\alpha_* \gamma_*)$, $R_2 = p(\alpha_* \beta_* \delta_*)$, and $R_3 = p(\alpha_* \beta_*)$. By adjoining the relation $\beta^2 = 1$ into the presentation for $\GG$, $p \circ *$ becomes an isomorphism, so we can see that 
\[ \GG \Big/  \left\llangle \beta^2 \right\rrangle \cong GL(2, \mathbb{Z}),\]
and the kernel of $p \circ *$ is $\left\llangle \beta^2 \right\rrangle$.

Now, $\varepsilon_* \notin \mathcal{D}$, but from the directly verifiable relations (induced by the relations in $\extG$ noted above)
\[ \es^2 = I, \quad \es \as = \as \es, \quad \es \bs = \as \bs \es, \quad  \es \gs = \gs \es, \quad \es \ds  = \ds^{-1} \es,\]
it can be shown that the image of $* : \extG \to GL(4, \mathbb{Z})$ is the extension of $\left\langle \es \, \, | \, \, \es^2 \right\rangle$ by $\mathcal{D}$,
\[ \mathcal{D} \rtimes \left\langle \es \, \, | \, \, \es^2 \right\rangle, \]
and that the kernel of $* : \extG \to \mathcal{D} \rtimes \left\langle \es \, \, | \, \, \es^2 \right\rangle$ is, again, $\left\llangle \beta^2 \right\rrangle$.

Thus,
\[ \extG \Big/ \left\llangle \beta^2 \right\rrangle \cong \set{ \begin{pmatrix} C & 0 \\ 0 & (C^{-1})^T \end{pmatrix} \, \, | \, \, C \in GL(2, \mathbb{Z}) } \rtimes \left\langle \es \, \, | \, \, \es^2 \right\rangle.\]

From this, we can conclude the following.
\begin{prop} \label{prop:Extended} Two curves $K$ and $K'$ are (homologically) related by an extended Goeritz group element if and only if they are (homologically) Goeritz equivalent, or if $K$ and $\varepsilon K'$ are (homologically) Goeritz equivalent.
\end{prop}

As the actions of the (non-extended) Goeritz group on homology act separately (though not independently) on the first two and the last two factors of the homology vector, the following lemmas will be of great use.

\begin{lem} \label{lem:GL2ZAction} Let $\textbf{\emph{v}} = (a, x)^T$, and $\textbf{\emph{v}}' = (a', x')^T$ be non-zero vectors in $\mathbb{Z}^2$. There exists a matrix $G \in GL(2, \mathbb{Z})$ so that $G\textbf{\emph{v}} = \textbf{\emph{v}}'$ if and only if $\gcd(a, x) = \gcd(a', x')$.
\end{lem}

\begin{proof}
First, suppose that $G = \begin{pmatrix} s & t \\ u & v \end{pmatrix} \in GL(2, \mathbb{Z})$ and $G\textbf{v} = \textbf{v}'$. Call $d = \gcd(a, x)$ and $d' = \gcd(a', x')$. By Bezout's Identity, there exist $e_1, e_2, e_1', e_2' \in \mathbb{Z}$ so that $e_1 a + e_2 x = d$, and $e_1' a' + e_2' x' = d'$. Now, $G\textbf{v} = \textbf{v}'$ implies that $a' = sa + tx$ and $x' = ua + vx$. Then,
\begin{eqnarray*} d' & = & e_1' a' + e_2' x' \\
& = & e_1' (sa + tx) + e_2' (ua + vx) \\
& = & (e_1's + e_2' u)a + (e_1' t + e_2' v)x. \end{eqnarray*}

Thus, by a corollary to Bezout's Identity, $d$ divides $d'$. Since $G$ is invertible, this argument can be reversed to show that $d'$ divides $d$, and hence $d = d'$.

Conversely, suppose that $\gcd(a, x) = \gcd(a', x') = d$. Then there exist $k, l, k', l' \in \mathbb{N}$ so that $a = kd$, $x = ld$, $a' = k'd$, and $x' = l'd$. Now, observe that $\gcd(k, l) = \gcd(k', l') = 1$, and there exist $e_1, e_2, e_1', e_2' \in \mathbb{Z}$ so that $e_1 k + e_2 l = 1 = e_1' k' + e_2' l' $.

Let $M = \begin{pmatrix} k & -e_2 \\ l & e_1 \end{pmatrix}$, which is evidently in $GL(2, \mathbb{Z})$. Then, note that 
\begin{eqnarray*} M \begin{pmatrix} d \\ 0 \end{pmatrix} & = & \begin{pmatrix} k & -e_2 \\ l & e_1 \end{pmatrix} \begin{pmatrix} d \\ 0 \end{pmatrix} \\
& = & \begin{pmatrix} kd  \\ ld \end{pmatrix} \\
& = & \begin{pmatrix} a \\ x \end{pmatrix}. \end{eqnarray*}

Similarly, $M' =  \begin{pmatrix} k' & -e_2' \\ l' & e_1' \end{pmatrix} \in GL(2, \mathbb{Z})$ and $M' \begin{pmatrix} d \\ 0 \end{pmatrix} = \textbf{v}'$, so $G = M' M^{-1}$ is the desired matrix.
\end{proof}

\begin{rem} \label{rem:Factorization} For each Goeritz matrix $M$ matrix carrying $[K]$ to $[K']$, the algorithm that produces the Smith Normal form of $M$ (see, e.g., Appendix C of \cite{ElmLAA}) provides a factorization into elementary matrices, which can then be factored directly into our generators $\as$, $\bs$, $\ds$, $\gs$, providing an element $g \in \GG$ inducing the matrix $M$.
\end{rem}

\begin{cor}
For any knot $K$ in $F$ with $[K] = (a, x, b, y)^T$, $K$ is Goeritz equivalent to a knot $K'$ with $[K'] = (d, 0, b', y')^T$ for some $b', y' \in \mathbb{Z}$, and $d = gcd(a, x)$.  
\end{cor}

\begin{proof}
Let $d = gcd(a, x)$. Since the Goeritz group acts on homology by $C(a, x)^T$, and $((C^{-1})^T) (b, y)^T$,  then by Lemma \ref{lem:GL2ZAction}, we can find the matrix $C$ that carries $(a, x)^T$ to $(d, 0)^T$, and then the corresponding $((C^{-1})^T)$ carries $(b, y)^T$ to some vector, call it $(b', y')^T$. Then, factorize the $4\times 4$ matrix $\begin{pmatrix} C & 0 \\ 0 & \left(C^{-1} \right)^T \end{pmatrix}$, as in Remark \ref{rem:Factorization}, into the corresponding Goeritz generators, and the actual curve K gets carried by this product to some new curve $K'$, so that $[K'] = (d, 0, b', y')^T$.
\end{proof}

\begin{lem}\label{lem:NonOrtho} If $\textbf{\emph{v}}_1 = (a, x)^T$ and $\textbf{\emph{v}}_2 = (b, y)^T$ are non-orthogonal vectors in $\mathbb{Z}^2$, and $\textbf{\emph{v}}'_1 = (a', x')^T$ and $\textbf{\emph{v}}_2' = (b', y')^T$ are any vectors in $\mathbb{Z}^2$, there are at most 
 two $4\times4$ Goeritz matrices $M$ over $\mathbb{Z}$ so that $M \cdot (a, x, b, y)^T = (a', x', b', y')^T$.
\end{lem}

\begin{proof} Suppose there is a matrix
$A = \begin{pmatrix} s & t \\ u & v \end{pmatrix}$, so that 
$M = \begin{pmatrix} A & 0 \\ 0 & \left(A^T \right)^{-1} \end{pmatrix}$ has the property that 
$M \cdot (a, x, b, y)^T = (a', x', b', y')^T.$

Since $\left(A^T \right)^{-1} = \frac{1}{\det(A)} \begin{pmatrix} v & -u \\ -t & s \end{pmatrix}$, we have 
\[ \frac{1}{\det(A)} \begin{pmatrix} v & -u \\ -t & s \end{pmatrix} \cdot \begin{pmatrix} b \\ y \end{pmatrix} = \begin{pmatrix} b' \\ y' \end{pmatrix}. \]

Together with the equations coming from $A \textbf{v}_1 = \textbf{v}'_1$, we have:

\begin{align*} \frac{(vb - uy) }{ d } = b' \\
\frac{( -tb + sy)}{d} = y' \\
sa + tx = a' \\
ua + vx = x',\end{align*}
where $d = \det(A)$. As we require that $A \in GL(2, \mathbb{Z})$, $d$ must be $\pm 1$.

This is equivalent to the matrix equation
$\begin{pmatrix} \frac{y}{d} & -\frac{b}{d} & 0 & 0 \\ a & x & 0 & 0 \\ 0 & 0 & \frac{-y}{d} & \frac{b}{d}  \\ 0 & 0 & a & x \end{pmatrix} \cdot \begin{pmatrix} s \\ t \\ u \\ v \end{pmatrix} = \begin{pmatrix} y' \\ a' \\ b' \\ x' \end{pmatrix}$.

As $(a, x)$ and $(b, y)$ are non-orthogonal, each of the $2 \times 2$ submatrices is invertible over $\mathbb{Q}$. So for each of $d = \pm 1$, there is a unique solution for $s, t, u,$ and $v$ over $\mathbb{Q}$.
In particular, for $d = 1$ and $d=-1$,
\[\begin{pmatrix} s & t \\ u & v \end{pmatrix} = \frac{1}{ab+xy} \begin{pmatrix} xy'd + a'b \ & -ay'd + a'y \\ -b'xd + bx' & ab'd + x'y \end{pmatrix}.\] 
Note that while these matrices will never be equal, either of them may be in $GL(2, \mathbb{Q})$ but not $GL(2, \mathbb{Z})$.
\end{proof}

In fact, the determinants of the matrices found in Lemma \ref{lem:NonOrtho} can be computed, showing that 
\[d = d \frac{a'b' + x'y'}{ab+xy}.\] This hints at the importance of the quantity $ab + xy$.

\subsection{Slope bases and the split product}
\label{subsection:SplitProduct}

\begin{defn}
Let $F$ be a genus $2$ Heegaard surface for $S^3$ bounding handlebodies $H$ and $H'$. Let $\textbf{k} \in \mathbb{Z}^4$ represent an element of $H_1(F, \mathbb{Z})$ with respect to the standard basis. Then the first $2$ components of $\textbf{k}$ determine an element $\textbf{k}_{H'} \in H_1(H') \cong \mathbb{Z}^2$ and the second $2$ components determine an element $\textbf{k}_H \in H_1(H) \cong \mathbb{Z}^2$. We may then compute the dot product of $\textbf{k}_{H'}$ and $\textbf{k}_H$, which we will call the \textit{split product} of $\textbf{k}$ (with respect to the standard basis), denoted $SP(\textbf{k})$. When $SP(\textbf{k}) = 0$, we say $\textbf{k}$ is \textit{split orthogonal with respect to $F$}. For a (non-separating) curve $K$ on $F$, we define $SP(K) = SP(\textbf{k})$, where $\textbf{k}$ is the homology representative of $K$ in $F$. 
\end{defn}

The following proposition is straightforward to show.

\begin{prop} \label{prop:SplitProductPreserved}
The split product is preserved by elements of the extended Goeritz group. That is, every curve in an extended Goeritz equivalence class has the same split product. 
\end{prop}

\begin{proof}
One need only check that each of the generating elements of the extended Goeritz group induces a map that preserves the split product of $\textbf{k}$. 
\end{proof}

However, the result turns out to be true for a deeper reason. The surface slope of a curve is certainly preserved by any extended Goeritz group action, and when the split product is computed using the standard basis, it provides an immediate way to calculate the surface slope.

\begin{thm}\label{thm:sp=ss}
Suppose $K \subset F$ is a non-separating simple closed curve. The split product of $[K]  \in H_1(F)$ is equal to the surface slope of $K$ with respect to $F$.
\end{thm}

\begin{proof}
A single meridian, say, representing homology class $(1, 0, 0, 0)^T$ clearly satisfies the conditions of the theorem. The mapping class group of the genus two surface can be generated by Dehn twisting on the five curves shown in Figure \ref{figure:generatingset}, and the mapping class group is transitive on isotopy classes of non-separating curves, so it suffices to show that the split product and the surface slope remain equal to each other after a Dehn twist along each of the five curves.

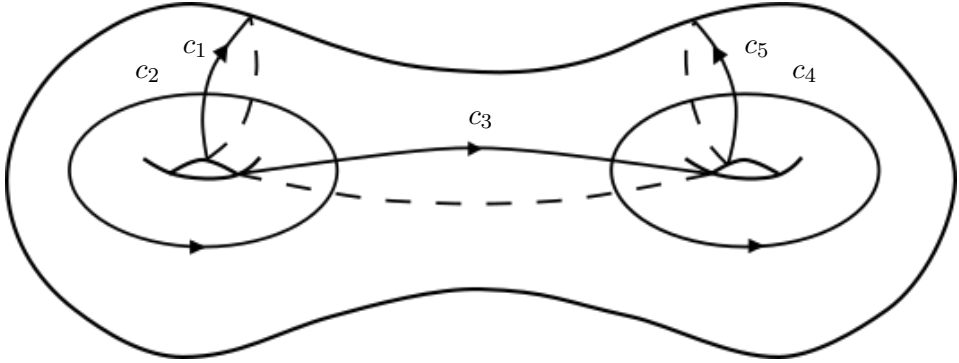
\begin{figure}
\begin{center}
\begin{tikzpicture}
    \node[anchor=south west,inner sep=0] (image) at (0,0) {\includegraphics[scale = .9]{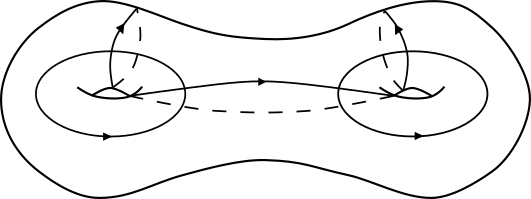}};
    \begin{scope}[
        x={(image.south east)},
        y={(image.north west)}
    ]
               \node [font=\bfseries] at (0.2,0.87) {$c_1$};
               \node [font=\bfseries] at (0.15,0.8) {$c_2$};
               \node [font=\bfseries] at (0.79,0.87) {$c_5$};
               \node [font=\bfseries] at (0.84,0.8) {$c_4$};
               \node [font=\bfseries] at (0.5,0.67) {$c_3$};
    \end{scope}
\end{tikzpicture}

\caption{A generating set for the mapping class group of $F$, together with a choice of orientation on each curve.}
\label{figure:generatingset}
\end{center}
\end{figure}

For the five curves, denoted $c_1$, $c_2$, $c_3$, $c_4$, and, $c_5$ in Figure \ref{figure:generatingset}, we will choose an orientation for each curve, as shown, so that the orientations of $c_1$, $c_2$, $c_4$, and $c_5$ agree with the choices of orientations on the standard generators for $H_1(F)$, namely $\textbf{a}$, $\textbf{b}$, $\textbf{y}$, and $\textbf{x}$, respectively. Further, orient $F$ so that the positive normal direction points into $H'$. Denote a push-off of a curve $c \subset F$ in the positive (resp., negative) direction as $c^+$ (resp., $c^-$).

\begin{rem}\label{rem:twistypoo} Performing a positive (left-handed) Dehn twist around $c_1$, $c_3$, or $c_5$ changes the surface slope of a curve $K$ by $-lk(K, c_j^+)^2$, while performing a positive Dehn twist around $c_2$ or $c_4$ changes the surface slope of $K$ by $lk(K, c_j^-)^2$.
\end{rem}

\begin{lem}\label{lem:intersectionnumber}
For a curve $K$ with $[K] = (a, x, b, y)^T$, $lk(K, c_3^+) =  b-y$.
\end{lem}

\begin{proof} 
Since $K = c_1$ (with either orientation) satisfies the statement of the lemma, and the mapping class group acts transitively on non-separating curves, it suffices to show that Dehn twisting a curve around a generator of the mapping class group will change each of the quantities $lk(K, c_3^+)$ and $b-y$ in the same way.

Suppose $K$ satisfies the statement. Dehn twisting around $c_1$, $c_3$, or $c_5$ has no effect on $b$, $y$, or the linking number. Dehn twisting around $c_2$ changes $b$ by  $\aint(K, c_2)$ and $y$ by 0, while $lk(K, c_3^+)$  changes by $\aint(K, c_2)$ as well. Similarly, Dehn twisting around $c_4$ changes $b$ by 0 and $y$ by $\aint(K, c_4)$, while also $lk(K, c_3^+)$ changes by $-\aint(K, c_4)$.
\end{proof}

Suppose that $[K] = (a, x, b, y)^T$, so the split product of $K$ is $a b + x y$ and the surface slope of $K$ with respect to $F$ is $ab + xy$. 
Dehn twisting along $c_1$ sends $[K]$ to $(a+\aint(K, c_1), x, b, y)^T$, so the split product of the resulting vector is $ab + xy + b\aint(K, c_1) = ab + xy - b^2$, since $\aint(K, c_1) = -b$. By Remark \ref{rem:twistypoo}, the surface slope changes by $-lk(K, c_1^+)^2$, which is equal to $-b^2$.
Dehn twisting along $c_2$ sends $[K]$ to $(a, x, b + \aint(K, c_2) , y)^T$, which has split product $ab + xy + a\aint(K, c_2) = ab + xy + a^2$, since $\aint(K, c_2) = a$. By Remark \ref{rem:twistypoo}, the surface slope changes by $lk(K, c_2^-)^2$, which is equal to $a^2$. 
Dehn twisting around $c_4$ and $c_5$ are the same, \emph{mutatis mutandis}.

Finally, Dehn twisting along $c_3$ sends $[K]$ to $(a + \aint(K, c_3), x - \aint(K, c_3), b, y)^T$, which has split product $ab + xy + (b-y)\aint(K, c_3) = ab + xy - (b-y)^2$, since $\aint(K, c_3) = y- b$. But $lk(K, c_3^+) = b-y$, by Lemma \ref{lem:intersectionnumber}, and by Remark \ref{rem:twistypoo}, the surface slope changes by $-lk(K, c_3^+)^2$.
\end{proof}

\begin{rem}
\label{rem:NotationConvention}
In previous work \cite{AmoDolRatACPTTK}, we used a different orientation convention for the generators $\textbf{b}$ and $\textbf{y}$, and showed that the twisted torus knot $K(p,q,r,n)$ had $[K]_F = (q, nr, -p, -r)^T$. With the current convention, $[K] = (q, nr, p, r)^T$, so $SP(K) = pq + nr^2$, which is exactly the surface slope of $K$ with respect to $F$. 
\end{rem}

Remark \ref{rem:NotationConvention} raises the question of just how much the split product depends on the choice of basis. To answer this, we will need to discuss symplectic bases and introduce what we will call a \emph{slope basis}.

A symplectic form, $\omega$, on a vector space is a non-degenerate, bilinear 2-form. A basis, $\set{\textbf{v}_1, \textbf{w}_2, \dots, \textbf{v}_n, \textbf{w}_n}$, for a (necessarily $2n$-dimensional) vector space with a symplectic form is called a \emph{symplectic basis} if $\omega(\textbf{v}_i, \textbf{v}_j) = \omega(\textbf{w}_i, \textbf{w}_j) = 0$ for all $i, j$, and $\omega(\textbf{v}_i, \textbf{w}_j) = \delta_{i j}$, the Kronecker delta. The \emph{standard symplectic form} on $\mathbb{R}^{2n}$ is
\[ \omega(\textbf{v}, \textbf{v}') = \sum\limits_{i=1}^n (v_i w_i' - v_i' w_i),\] for $\textbf{v} = (v_1, w_1, \dots, v_n, w_n)^T$, $\textbf{v}' = (v_1', w_1', \dots, v_n', w_n')^T$, and it is unique up to change of basis of $\mathbb{R}^{2n}$.

Now, the algebraic intersection number, 
\[ \aint : H_1(F; \mathbb{Z}) \wedge H_1(F; \mathbb{Z}) \to \mathbb{Z} \]
extends uniquely to a symplectic form on $H_1(F; \mathbb{R}) \cong \mathbb{R}^4$. A \emph{geometric symplectic basis} is a collection of oriented simple closed curves that form a symplectic basis. Such a geometric symplectic basis is formed by re-ordering the curves of our standard basis from Figure \ref{fig:goeritz}, 
\[ \set{ \textbf{a}, \textbf{b}, \textbf{x}, \textbf{y} }.\]

For more, see Chapter 6 of \cite{FarMarPMCG}.

There is an additional property that the standard basis has that relates the homology of the surface $F$ to the topologies of the handlebodies $H$ and $H'$, namely that $\textbf{a}$ and $\textbf{x}$ bound disks in $H$, and $\textbf{b}$ and $\textbf{y}$ bound disks in $H'$.

\begin{defn} We say a collection of simple closed curves, $\set{a', x', b', y'}$, on the genus 2 Heegaard surface of $S^3$ is a \emph{\goodbasis} if $a'$ and $x'$ bound disks in $H$, $b'$ and $y'$ bound disks in $H'$, and $\set{a', b', x', y'}$ is a geometric symplectic basis. 
\end{defn}

\SplitProductSurfaceSlope*

\begin{proof}
Let $\beta' = \set{a', x', b', y'}$ be any \goodbasis \ and let $\beta = \set{\textbf{a}, \textbf{x}, \textbf{b}, \textbf{y}}$ be the standard basis for $H_1(F)$. Since $a'$ bounds a disk in $H$, it is trivial in $H_1(H)$, while $\textbf{a}$ and $\textbf{x}$ are generators for $H_1(H')$, so $[a'] = s\textbf{a} + u\textbf{x}$, for some $s, u \in \mathbb{Z}$. Similarly, $[x'] = t\textbf{a} + v\textbf{x}$, $[b'] = e\textbf{b} + g\textbf{y}$, and $[y'] = f\textbf{b} + h\textbf{y}$ for some $t, v, e, f, g, h \in \mathbb{Z}$.

Thus, the change of basis matrix from $\beta$ to $\beta'$ is of the form $\begin{pmatrix} C & 0 \\ 0 & D \end{pmatrix}$,
with $C = \begin{pmatrix} s & t \\ u & v \end{pmatrix}$, and $D = \begin{pmatrix} e & f \\ g & h \end{pmatrix}$.

Further, as $\set{a', b', x', y'}$ is a geometric symplectic basis, using coordinates of a vector with respect to the standard symplectic basis, we have

\begin{small}
\[1 = \aint(a', b') = \omega([a'], [b']) = \omega((s, 0, u, 0)^T, (0, e, 0, g)^T) = se + ug = (s, u) \cdot (e, g)^T, \]
\[1 = \aint(x', y') = \omega([x'], [y']) = \omega((t, 0, v, 0)^T, (0, f, 0, h)^T) = tf + vh = (t, v) \cdot (f, h)^T,\]
\[0 = \aint(a', y') = \omega([a'], [y']) = \omega((s, 0, u, 0)^T, (0, f, 0, h)^T) = sf + uh = (s, u) \cdot (f, h)^T,\]
\[0 = \aint(b', x') = \omega([b'], [x']) = \omega((0, e, 0, g)^T, ((t, 0, v, 0)^T) = -et - gv = -(t, v) \cdot(e, g)^T.\] \end{small}

These conditions precisely ensure that $D = \left( C^T \right)^{-1}$, so that the change of basis matrix is a Goeritz matrix. Then, Proposition \ref{prop:SplitProductPreserved} guarantees that $SP([K]_{\beta'}) = SP([K]_\beta)$.
\end{proof}

We can now prove Theorem \ref{thm:Homology}.

\TheoremHomology*

\begin{proof} With respect to any slope basis, let $[K] = (a, x, b, y)^T$ and $[K'] = (a', x', b', y')^T$ be the homology vectors of the curves. By Theorem \ref{thm:SplitProductSurfaceSlope}, the surface slope of $K$ is equal to its split product, which establishes the necessity and sufficiency of condition (\ref{cond:slopes}). The necessity and sufficiency of conditions (\ref{cond:firstgcd}) and (\ref{cond:secondgcd}) follow from Lemma \ref{lem:GL2ZAction}. Finally, since the split product is non-zero, the two vectors $(a, x)^T$ and $(b, y)^T$ are non-orthogonal, and by Lemma \ref{lem:NonOrtho}, there are at most two matrices carrying $[K]$ to $[K']$. Conditions (\ref{cond:topleft} - \ref{cond:bottomright}) (resp., (\ref{cond:topleftprime} - \ref{cond:bottomrightprime})) correspond to the matrix over $\mathbb{Q}$ when $d = 1$ (resp., $d = -1$) actually residing in $GL(4, \mathbb{Z})$.
\end{proof}

Proposition \ref{prop:Extended} allows us to easily extend this obstruction to the extended Goeritz group by noting that $[\varepsilon K'] = (y', b', x', a')^T$.

\subsection{Action on fundamental groups}
\label{subsection:FundamentalGroups}
If an essential (separating) curve on the Heegaard surface $F$ bounds a disk in each handlebody, then the union of these disks is a reducing sphere for the Heegaard splitting. We will call such a curve a \emph{reducing sphere curve}, and a Dehn twist around such a curve a \emph{reducing sphere twist}.

\begin{thm}\label{thm:ReducingSphereTwists} Suppose $K$ and $K'$ are homologous simple closed curves with non-zero surface slopes. Then $K$ and $K'$ are Goeritz equivalent if and only if there is a sequence of reducing sphere twists, $T_1$, $T_2$, \dots, $T_n$ so that $K' = (T_n \circ \cdots \circ T_2 \circ T_1)(K)$. 
\end{thm}

\begin{proof}
Recall that $\beta^2$ is a Dehn twist around the belt curve, which is a reducing sphere curve. Conjugating $\beta^2$ by elements of the Goeritz group will result in another reducing sphere twist. In fact, since the Goeritz group acts transitively on reducing spheres (\cite{SchaA3SPG2HS}), $\llangle \beta^2 \rrangle$ is the group of all reducing sphere twists, contained within the Goeritz group.

Conversely, if $K$ and $K'$ are related by a Goeritz group element, as they are homologous, they must be related by $ker(p \circ *) = \llangle \beta^2 \rrangle$, as per the discussion in Section \ref{subsection:ActionHomology}.
\end{proof}

We can now provide the second obstruction to Goeritz equivalence as a corollary to Theorem \ref{thm:ReducingSphereTwists}.

\TheoremHomotopy*

\begin{proof}
The curve $K$ represents elements of each of the two free groups $\pi_1(H)$ and $\pi_1(H')$, up to free homotopy, by pushing $K$ slightly off of $F$ into either of $H$ or $H'$. Evidently, a reducing sphere twist of $K$ will not change these free homotopy classes. 
\end{proof}

\section{Examples and applications}\label{section:examples}

\subsection{The condition of non-zero surface slope}
\label{subsection:surfaceslope}

The hypothesis that $K$ and $K'$ have non-zero surface slope in Theorem \ref{thm:Homology} is important to rule out more unpredictable behavior. For instance, consider the trite example that $K = K' = \textbf{a}$, the first generator in the standard basis, so that $\textbf{k} = [K] = \textbf{k}' = [K'] = (1, 0, 0, 0)^T$. Then conditions (\ref{cond:firstgcd}) and (\ref{cond:secondgcd}) from Theorem \ref{thm:Homology} do not even make sense. In this case, the identity certainly sends $K$ to $K'$. But Lemmas \ref{lem:GL2ZAction} and \ref{lem:NonOrtho} do not apply, and we can see that additionally, for instance, the matrices $A_{\mu} = \begin{pmatrix} 1 & \mu \\ 0 & -1 \end{pmatrix}$ gives rise to infinitely many candidate Goeritz-form matrices, and infinitely many candidate Goeritz elements, $\beta(\gamma \delta \gamma)^\mu$, for $\mu \in \mathbb{Z}$, fixing $\textbf{k}$. In fact, all of these homeomorphisms do fix the curve $K$ as well, but the infinite list exhibits a theoretical problem with any kind of systematic search. 

Or slightly less trivially, let $K = \textbf{y}$, the last generator in the standard basis, so that $\textbf{k} = [K] = (0, 0, 0, 1)^T$, and let $K'$ be the curve resulting from the band sum of $\textbf{y}$ and $\textbf{b}$ with the opposite orientation, so that $\textbf{k}' = [K'] = (0, 0, -1, 1)^T$. In this case, condition (\ref{cond:secondgcd}) makes sense, but condition (\ref{cond:firstgcd}) does not. Certainly, $\ds \textbf{k} = \textbf{k}'$. However, $\delta(K) \neq K'$, since the two curves can be distinguished in the fundamental group. Nonetheless, there exist infinite families of Goeritz group elements that induce homology automorphisms carrying $\textbf{k}$ to $\textbf{k}'$, and, in fact, carrying $K$ to $\delta(K) \neq K'$, such as $\delta \gamma \delta^n \gamma$ for any $n \in \mathbb{Z}$.

Even when conditions (\ref{cond:firstgcd}) and (\ref{cond:secondgcd}) are both meaningful, if the split product/surface slope is zero, Lemma \ref{lem:NonOrtho} will not ensure a finite list of candidate matrices, so there is no simple set of conditions to check in general. Theorem \ref{thm:Homology} is also difficult to apply directly when investigating an infinite family of curves determined by parameters. However, the techniques used here can still provide new information, as we demonstrate with the twisted torus knots $K(r^2, q, r, -q)$, with $\gcd(r,q)=1$, $1\le q < r^2$. These are non-trivial as long as $(q, r) \neq (1, 2)$ (see \cite{LeeTTKU}), and we will show that they can be distinguished up to extended Goeritz equivalence despite having zero surface slope. 

Twisted torus knots are defined in \cite{DeaSSFDSHK}. See \cite{AmoDolRatACPTTK} for the notation used here, as well as this relevant lemma, noting, however, that the third and fourth components of the homology vector differ by a sign from the lemma as stated there, as in Remark \ref{rem:NotationConvention}.

\begin{lem}\label{lem:homologyvector}
In $H_1(F)$, the twisted torus knot $K(p,q,r,n)$ is represented by the class $(q, nr, p, r)^T$.
\end{lem}

\begin{prop}
No two of the twisted torus knots $K(r^2, q, r, -q)$, with $\gcd(r,q)=1$, $1\le q < r^2$, and $ r \ge 2$,  share an extended Goeritz equivalence class.
\end{prop}

\begin{proof}
Consider $K_1 = K(r_1^2, q_1, r_1, -q_1)$ and  $K_2 = K(r_2^2, q_2, r_2, -q_2)$, with the appropriate requirements on $r_i$ and $q_i$ for $i = 1, 2$. We have $\textbf{k}_1 = (q_1, -q_1 r_1,r_1^2, r_1)^T$ and $\textbf{k}_2 = (q_2, -q_2 r_2,r_2^2, r_2)^T$. By Lemma \ref{lem:GL2ZAction}, there exists a matrix sending $(q_1, -q_1 r_1)^T$ to $(q_2, -q_2 r_2)^T$ if and only if $q_1 = q_2$, and there exists a matrix sending $(r_1^2, r_1)^T$ to $(r_2^2, r_2)^T$ if and only if $r_1 = r_2$.

On the other hand, consider, $\es \textbf{k}_2 = (r_2, r_2^2, -q_2r_2, q_2)^T$. Now, Lemma \ref{lem:GL2ZAction} says there exists a matrix sending $(q_1, -q_1 r_1)^T$ to $(r_2, r_2^2)^T$ if and only if $q_1 = r_2$, and there exists a matrix sending $(r_1^2, r_1)^T$ to $(-q_2r_2, q_2)^T$ if and only if $r_1 = q_2$. Then, we are seeking a matrix in $GL(2, \mathbb{Z})$ that satisfies
\begin{equation} \label{eqn:ZeroSlopeEpsilonMatrixNonReplacement} \begin{pmatrix} s & t \\ u & v \end{pmatrix} \begin{pmatrix} q_1 \\ -q_1r _1\end{pmatrix} = \begin{pmatrix} r_2 \\ r_2^2 \end{pmatrix}  
\end{equation}
\begin{equation} \label{eqn:ZeroSlopeEpsilonTransposeInverseNonReplacement}
\Delta \begin{pmatrix} v & -u \\ -t & s \end{pmatrix} \begin{pmatrix} r_1^2 \\ r_1 \end{pmatrix} = \begin{pmatrix} -q_2r_2 \\ q_2 \end{pmatrix}.
\end{equation}

Here, $\Delta= sv - tu (= \pm 1)$ is the determinant of the sought matrix.

With the replacements $r_2 = q_1$ and $r_1 = q_2$, these become:

\begin{equation} \label{eqn:ZeroSlopeEpsilonMatrix} \begin{pmatrix} s & t \\ u & v \end{pmatrix} \begin{pmatrix} q_1 \\ -q_1q _2\end{pmatrix} = \begin{pmatrix} q_1 \\ q_1^2 \end{pmatrix}  
\end{equation}
\begin{equation} \label{eqn:ZeroSlopeEpsilonTransposeInverse}
\Delta \begin{pmatrix} v & -u \\ -t & s \end{pmatrix} \begin{pmatrix} q_2^2 \\ q_2 \end{pmatrix} = \begin{pmatrix} -q_1q_2 \\ q_2 \end{pmatrix}.
\end{equation}
Then (\ref{eqn:ZeroSlopeEpsilonMatrix}) gives $s - tq_2 = 1$ and $u - v q_2 = q_1$. On the other hand, from (\ref{eqn:ZeroSlopeEpsilonTransposeInverse}), we have $\Delta (vq_2 - u) = q_1$ and $\Delta(s - t q_2) = 1$. Substituting the first two resulting equations into the second two resulting equations, we find that $\Delta$ must be both $1$ and $-1$, a contradiction.
\end{proof}

\subsection{An example applying the homology obstruction}
\label{subsection:homologysuff}

Here we discuss a family of examples, which are twisted torus knots, for which the homology obstruction is sufficient to show two curves are not related by an extended Goeritz element. This example provides a more streamlined proof of previous results of the authors, with E. Amoranto.

In \cite{AmoDolRatACPTTK}, the authors, along with E. Amoranto, prove the following:

\begin{thm}
\label{thm:AdditionalCasesPositiveTwistedTorusKnots}
For integers $k$, $q$ and $m$ with $q>2$, $1 \le m < q$, $\gcd(q, m) = 1$, and $k = 0$ or $k\ge 2$ , the twisted torus knots $K_1 = K(kq + m, q, m, -1)$ and $K_2 = K(kq+ q- m, q, q-m, -1)$ are not related by an extended Goeritz element. 
\end{thm}

The result was proven for $k \geq 2$ by the first author in \cite{DolFPSTTK}, and in \cite{AmoDolRatACPTTK}, the authors prove it for $k = 0 $. In that work, the authors suspected that the knots would be related by an extended Goeritz element when $k=1$. As another application, we demonstrate that the homology obstruction suffices in the case of $k =0 $, and later investigate the remaining case that $k =1$, showing that the additional homotopy obstruction is not sufficient to draw a final conclusion.

\begin{proof}[Proof of Theorem \ref{thm:AdditionalCasesPositiveTwistedTorusKnots} when $k=0$.]
Using Lemma \ref{lem:homologyvector}, we see that, for $k=0$, the homology representatives of $K_1$ and $K_2$ are $\textbf{k}_1 = (q, -m, m, m)^T$ and $\textbf{k}_2 = (q, m-q , q - m, q-m)^T$, respectively. 

Let $d = \gcd(q, -m) = 1 = \gcd(q, m-q) = d'$, and $f = \gcd(m, m) = m$, $f' = \gcd(q-m, q-m) = q-m$. Thus, in order for $f$ to equal $f'$, as $m$ and $q$ are relatively prime, $m = 1$, so that $q = 2$, which is excluded by hypothesis.

If instead we consider $\es \textbf{k}_2 = (q-m, q-m, m-q, q)^T$, we calculate $d' = \gcd(q-m, q-m) = q-m$, and $f' = \gcd(m-q, q) = 1$. Then $d = d'$ implies that $1 = q-m$ and $f = f'$ implies that $m = 1$, which again forces $q = 2$. 
\end{proof}

\subsection{An example applying both homology and homotopy obstructions}
\label{subsection:homologyhomotopy}

Let us investigate the case that $k=1$ for the knots from Theorem \ref{thm:AdditionalCasesPositiveTwistedTorusKnots}. Consider the standard representation of $K_1$, shown in Figure \ref{fig:K_1}. The black bar on the left indicates the merging of the $q$ thinner arcs marked with a single arrow, which we will refer to as \emph{blue} arcs, with the $m$ thicker arcs marked with double arrows, which we will refer to as \emph{red} arcs. These are really two parts of the same curve. We call the black bar in this figure a \textit{switch}.

Figure \ref{fig:K_2} shows the standard representation of $K_2$ as a curve with a switch, with the $q$ thinner single-arrowed arcs, referred to as \emph{purple}, merging with the $q-m$ thicker double-arrowed arcs, referred to as \emph{pink}.

That the knots are isotopic and both positions are primitive/primitive with identical surface slopes is proven in \cite{AmoDolRatACPTTK}. In this specific example, Kenneth Baker showed that it is actually possible to view the knots $K_1$ and $K_2$ as sitting on the boundary of a thickened genus 1 page $P$ of a fibered knot in such a way that pushing $P$ through the monodromy carries $K_1$ to $K_2$, exchanging the two handlebodies bounded by two copies of $P$, verifying that they are, in fact, related by an extended Goeritz element \cite{BakPC}. Nonetheless, we employ the methods of this paper in the context of this example to demonstrate potential applications.

\begin{center}
\begin{figure}[h]
\begin{tikzpicture}
    \node[anchor=south west,inner sep=0] (image) at (0,0) {\includegraphics[width=\linewidth]{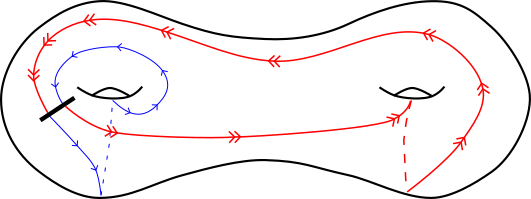}};
    \begin{scope}[
        x={(image.south east)},
        y={(image.north west)}
    ]
        \node [font=\bfseries] at (0.1,0.31) {$q$};
        \node [font=\bfseries] at (0.17,0.44) {$m$};
    \end{scope}
\end{tikzpicture}
\caption{The curve $K_1$, represented with a switch. That is, the sub-arcs labeled $q$ (thin, single-arrow, blue) and $m$ (thick, double-arrow, red) represent $q$ and $m$ parallel strands, respectively, so that at the black bar, $q+m$ strands connect with $m+q$ strands.}\label{fig:K_1}
\end{figure}
\end{center}

\begin{center}
\begin{figure}[h]
\begin{tikzpicture}
    \node[anchor=south west,inner sep=0] (image) at (0,0) {\includegraphics[width=\linewidth]{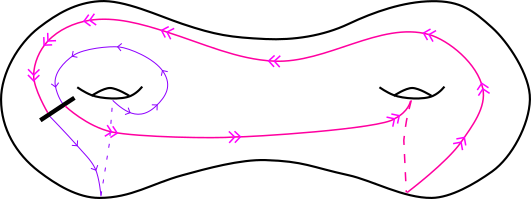}};
    \begin{scope}[
        x={(image.south east)},
        y={(image.north west)}
    ]
        \node [font=\bfseries] at (0.1,0.31) {$q$};
        \node [font=\bfseries] at (0.28,0.27) {$q-m$};
    \end{scope}
\end{tikzpicture}
\caption{The curve $K_2$, represented with a switch. Here, the sub-arcs labeled $q$ (thin, single-arrow, purple) and $q-m$ (thick, double-arrow, pink) represent $q$ and $q-m$ parallel strands, respectively, so that at the black bar, $q + (q-m)$ strands connect with $(q-m) + q$ strands.}
\label{fig:K_2}
\end{figure}
\end{center}

\subsubsection{The homology obstruction is inconclusive}
\label{subsub:homology}


For the curves $K_1$ and $K_2$, the homology representatives are, respectively, $\textbf{k}_1 = (q, -m, q + m, m)^T$ and $\textbf{k}_2 = (q, m-q , 2q - m, q-m)^T$. Then, $\gcd(q,m) = 1 = \gcd(q, m-q)$, $ \gcd(q+m, m) = 1 = \gcd(2q-m, q-m)$, and $SP(\textbf{k}_1) = q^2 + mq -m^2 = SP(\textbf{k}_2)$. We find no immediate obstruction.

The vectors $(q, -m)^T$ and $(q + m, m)^T$ are non-orthogonal, as $\gcd(q,m)= 1$. Similarly, $(q, m-q)^T$ and $(2q-m, q-m)^T$ are non-orthogonal, so there are at most two Goeritz matrices carrying $\textbf{k}_1$ to $\textbf{k}_2$. As our vectors are defined in terms of parameters, the forms of the matrices from Lemma \ref{lem:NonOrtho} are not immediately helpful. So we must dig a bit deeper into the question of when a given matrix over $\mathbb{Q}$ is invertible over $\mathbb{Z}$. Suppose $M \in GL(2, \mathbb{Z})$ carries $\textbf{k}_1$ to $\textbf{k}_2$ and denote $\det(M)$ by $\Delta$. Let $M = \begin{pmatrix} s & t \\ u & v \end{pmatrix}$. Then 
\[\begin{pmatrix} s & t \\ u & v \end{pmatrix} \begin{pmatrix} q \\ -m \end{pmatrix} = \begin{pmatrix} q \\ m-q \end{pmatrix},\] and we get two equations, $sq - tm = q$ and $uq - vm = m-q$. So let $l = \frac{t}{q} = \left(\frac{s-1}{m}\right)$, and $h = \left(\frac{v+1}{q}\right) = \left(\frac{u+1}{ \phantom{q} m \phantom{q} }\right)$.

Then,
\[ s = lm + 1, \qquad t = lq, \qquad u = hm -1, \qquad v = hq -1.\]
In order for $s, t, u, v \in \mathbb{Z}$, since $m$ and $q$ are relatively prime, we find that $l, h \in \mathbb{Z}$.

 $M = \begin{pmatrix} lm + 1 & lq \\ hm - 1 & hq -1 \end{pmatrix}$. 

Now we consider $\left( M^T \right)^{-1} = \Delta \begin{pmatrix} hq -1 & 1-hm \\ -lq & lm + 1 \end{pmatrix} $. We also have that 
\[\Delta \begin{pmatrix} hq -1 & 1-hm \\ -lq & lm + 1 \end{pmatrix} \begin{pmatrix} q +m \\ m \end{pmatrix} = \begin{pmatrix} 2q - m \\ q-m \end{pmatrix}.\] Using the second row, we have that $-lq(q+m)+ (lm+1)m = \Delta(q-m)$, which gives $m - lq^2 = (lm+\Delta)(q-m)$. If $l > 0$, the left hand side of this equation is negative because $q>m$, while the right hand side is at least 0. If $l = 0$, it must be that $m=\Delta(q-m)$, which either says that $q=0$ or $q =2m$, both of which are impossible. Finally, if $l<0$, then $m - lq^2>0$, so $lm+\Delta>0$ as well. Since $l<0$, $m>0$, and $\Delta = \pm 1$, there are no possible solutions to this inequality. 

Hence, we see that $K_1$ and $K_2$ are not Goeritz equivalent. However, we might still check for elements of the Goeritz group that send $(q, -m, q + m, m)^T$ to $\varepsilon_*(q, m-q , 2q - m, q-m)^T = (q-m, 2q-m, m-q, q)^T$.


To that end, we compute,  $ \gcd(q-m, 2q-m) = 1 = \gcd(q, -m)$, $\gcd(m-q, q) = 1 = \gcd(q+m,m)$, and of course $SP(\varepsilon_*(\textbf{k}_2)) = q^2 + qm -m^2 = SP(\textbf{k}_1)$. Finding again no immediate obstruction, we suppose $M = \begin{pmatrix} s & t \\ u & v \end{pmatrix} \in GL(2, \mathbb{Z})$, with $\det(M) = \Delta$, carries $\textbf{k}_1$ to $\es \textbf{k}_2$. Then,
\[\begin{pmatrix} s & t \\ u & v \end{pmatrix} \begin{pmatrix} q \\ -m \end{pmatrix} = \begin{pmatrix} q-m \\ 2q -m  \end{pmatrix},\]
and, $(s-1) q = (t-1) m$ and $(u-2)q = (v-1)m$. 

So let $l = \left(\frac{t-1}{q}\right) = \left(\frac{s-1}{ \phantom{q} m\phantom{q} }\right)$, and $h = \left(\frac{v-1}{q}\right) = \left(\frac{u-2}{ \phantom{q} m \phantom{q} }\right)$.
Then,
\[ s = lm + 1, \qquad t = lq+1, \qquad u = hm +2, \qquad v = hq +1.\]
Again, in order for $s, t, u, v \in \mathbb{Z}$, since $m$ and $q$ are relatively prime, we find that $l, h \in \mathbb{Z}$. Now we consider $\left(M^T \right)^{-1} = \Delta \begin{pmatrix} hq +1 & -(hm+2) \\ -(lq+1) & lm + 1 \end{pmatrix} $.

From the first row of 
\[\Delta \begin{pmatrix} hq +1 & -(hm+2) \\ -(lq+1) & lm + 1 \end{pmatrix} \begin{pmatrix} q +m \\ m \end{pmatrix} = \begin{pmatrix} m-q \\ q \end{pmatrix},\]

 $(hq+1)(q+m) - (hm+2)m = \Delta(m-q)$. After some rearrangement of this equation, we find that $(q-m)(hm+1 + \Delta) = -hq^2$. Since the sign of $-hq^2$ is determined by the sign of $h$, we consider $h>0$, $h<0$, and $h = 0$. 
 
 If $h>0$, then $-hq^2<0$ while $(q-m)(hm+1 + \Delta)>0$, a contradiction. If $h<0$, then $-hq^2>0$. Since $m<q$, $q-m$ is always positive, so the sign of $(q-m)(hm+1 + \Delta)$ depends only on the sign of $hm + 1 + \Delta$. This expression is positive exactly when $hm > -1 - \Delta$. Since $h<0$, this forces $\Delta = 1$, $m = 1$, and $h = -1$. With these considerations, we rewrite $M$  as $\begin{pmatrix} l+1 & lq+1 \\ 1 & 1-q \end{pmatrix} $. Since we also have that $\Delta = 1$, we obtain the additional equation $l - q - 2lq = 1$, which does not allow for $l$ and $q$ to be integers simultaneously while $q > 2$. (To see this, for instance, observe that $l = \frac{1 + q }{1 - 2q}$ is an increasing function of $q$ with value $-1$ at $q = 2$, and $l < 0$ for all $q > 2$.)
   
 Finally, suppose $h = 0$, so that $(q-m)(1+\Delta) = 0$. Since $q$ and $m$ are relatively prime, we have $\Delta = -1$. Now, rewriting $M$, we have
 \[ M = \begin{pmatrix} lm + 1 & lq+1 \\ hm +2 & hq +2 \end{pmatrix} = \begin{pmatrix} lm + 1 & lq+ 1 \\ 2 & 1 \end{pmatrix}.\]
 But the determinant restriction then says that $l(m-2q) = 0$, which forces $l = 0$, as $q$ and $m$ are relatively prime. 
 
Thus, we find exactly one candidate to give rise to a Goeritz matrix,
\[M = \begin{pmatrix} 1 & 1 \\ 2 & 1 \end{pmatrix}.\]

This matrix factors (see Remark \ref{rem:Factorization}) into $p( \as \bs \gs \ds^{-1} \gs \ds^2 ).$
 
This means 
\[ \as \bs \gs \ds^{-1} \gs \ds^2 [K_1] = \es [K_2]  .\] Let $g = \alpha \beta \gamma \delta^{-1} \gamma \delta^{2} $. (It is worth noting that this differs subtly from the map listed at the end of \cite{AmoDolRatACPTTK}, due to changes in conventions.) 

From this analysis then, we learn that $K_1$ and $K_2$ are not homologically Goeritz equivalent, but that they are related by an extended Goeritz element, and specifically that $g(K_1)$ and $\varepsilon(K_2)$ are homologous. So, $K_1$ and $K_2$ are related by an extended Goeritz element if and only if $g(K_1)$ and $\varepsilon(K_2)$ are related by a sequence of reducing sphere twists, as in Theorem \ref{thm:ReducingSphereTwists}.

\subsubsection{The homotopy obstruction is inconclusive}
\label{subsub:HomotopyInconclusive}

Finally, we explore the homotopy obstruction. We begin with a generally helpful observation about words in the fundamental group of a genus 2 handlebody representing curves on the boundary.

\begin{lem} \label{lem:cyclicpallindrome}
If $w$ is a word in the fundamental group of a genus 2 handlebody $H$ representing a curve $K$ on the boundary, with $K$ not null-homologous, then the word is cyclically palindromic, i.e., $w$ is cyclically equivalent to its reverse.
\end{lem}

\begin{proof}
Suppose $w = w_n w_{n-1} \cdots w_2 w_1$ is the word representing a curve $K$, which is not null-homologous, up to cyclic permutation, with respect to a pair of curves $c$ and $d$ each bounding a disk in $H$. Consider $\alpha$ as a homeomorphism of the handlebody, so that $\alpha$ induces an isomorphism on $\pi_1(H)$. Specifically, if we let $c' = \alpha(c)$, and $d' = \alpha(d)$, then $\alpha_*(w) = \alpha_*(w_n) \alpha_*(w_{n-1}) \cdots \alpha_*(w_2) \alpha_*(w_1)$ represents the pattern of intersections between $\alpha(K)$ and the two disks $c'$ and $d'$. Now, $\alpha$ simply reverses the orientation of every curve on the boundary of $H$. So on the one hand, this means that $c'$ and $d'$ are just the same curves $c$ and $d$, respectively, with opposite orientations, so that $\alpha_*(w) = w_n^{-1} w_{n-1}^{-1} \cdots w_2^{-1} w_1^{-1}.$ But on the other hand, $\alpha$ also reverses the orientation of $K$, so that $\alpha_*(w) = w_1^{-1} w_2^{-1} \cdots w_{n-1}^{-1} w_n^{-1}$. 
\end{proof}

Next, we establish some patterns about the curves $g(K_1)$ and $\varepsilon(K_2)$ on the surface. 

As a transitional state, Figure \ref{fig:gammadeltasquaredK_1} shows $K_1$ after $\gamma \delta^2$ has been applied, and Figure \ref{fig:finalK_1} shows $K_1$ after $g = \alpha \beta \gamma \delta^{-1} \gamma \delta^2$ has been applied. 

\begin{center}
\begin{figure}[h]
\begin{tikzpicture}
    \node[anchor=south west,inner sep=0] (image) at (0,0) {\includegraphics[width=\linewidth]{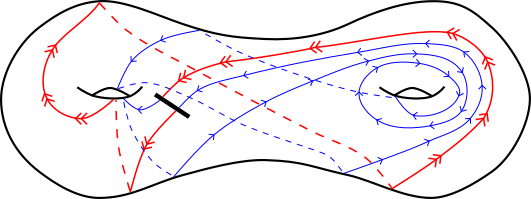}};
    \begin{scope}[
        x={(image.south east)},
        y={(image.north west)}
    ]
        \node [font=\bfseries] at (0.32,0.35) {$m$};
        \node [font=\bfseries] at (0.29,0.43) {$q$};
    \end{scope}
\end{tikzpicture}
\caption{The curve $\gamma\delta^2(K_1)$.}\label{fig:gammadeltasquaredK_1}
\end{figure}
\end{center}

\begin{center}
\begin{figure}[h]
\begin{tikzpicture}
    \node[anchor=south west,inner sep=0] (image) at (0,0) {\includegraphics[width=\linewidth]{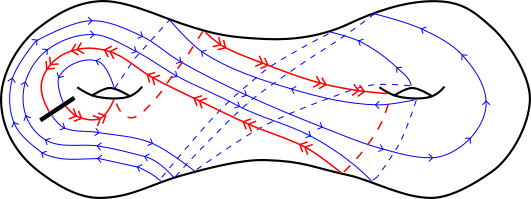}};
    \begin{scope}[
        x={(image.south east)},
        y={(image.north west)}
    ]
        \node [font=\bfseries] at (0.095,0.53) {$m$};
        \node [font=\bfseries] at (0.13,0.57) {$q$};
    \end{scope}
\end{tikzpicture}
\caption{The curve $g(K_1) = \alpha\beta\gamma\delta^{-1}\gamma\delta^2 (K_1)$.}\label{fig:finalK_1}
\end{figure}
\end{center}

Figure \ref{fig:epsK_2} shows $K_2$ after $\varepsilon$ has been applied. Note that $\as \bs \gs \ds^{-1} \gs \ds^2 [K_1] = \es [K_2] = (q-m,2q-m,m-q, q)^T$, as expected. 

\begin{center}
\begin{figure}[h]
\begin{tikzpicture}
    \node[anchor=south west,inner sep=0] (image) at (0,0) {\includegraphics[width=\linewidth]{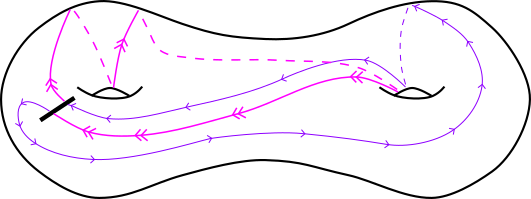}};
    \begin{scope}[
        x={(image.south east)},
        y={(image.north west)}
    ]
        \node [font=\bfseries] at (0.12,0.325) {$q-m$};
        \node [font=\bfseries] at (0.18,0.46) {$q$};
    \end{scope}
\end{tikzpicture}
\caption{The curve $\varepsilon ( K_2 )$.}\label{fig:epsK_2}
\end{figure}
\end{center}

For any specific values of $q$ and $m$, checking the homotopy obstruction is straightforward. Doing so for the entire family, however, entails capturing the homotopy behavior for each infinite family of curves. Nonetheless, a surprisingly involved, but hopefully enlightening, set of calculations show that there is no homotopy obstruction to the curves being Goeritz equivalent.

Given the pair of relatively prime positive integers $q$ and $m$ with $1 \leq m < q$. Let $d$ be the quotient and $r$ the remainder when $q$ is divided by $m$, so that $q = dm + r$ with $r < m$ (note that $0 < r$ since $\gcd(q, m) = 1$). Let $\widetilde{m} = q-m$, and define $\widetilde{d}$ and $\widetilde{r}$ to be the quotient and remainder when $q$ is divided by $\widetilde{m}$. Observe that if $m < \frac{q}{2}$, then $d \geq 2$, $m < (q-m)$, so that $\widetilde{d} = 1$, $\widetilde{r} = m$, while if $m > \frac{q}{2}$, then $d = 1$, $r = \widetilde{m}$.

Define the following four sequences as follows:

\begin{enumerate}

\item[($z_n$)] For each $-1 \leq n \leq m-1$, let $z_n$ be the smallest integer so that $m z_n \geq (n+1) r$. 

\item[($\widetilde{z}_n$)] For each $-1 \leq n \leq q-m-1$, let $\widetilde{z}_n$ be the smallest integer so that $(q-m) \widetilde{z}_n \geq (n+1) \widetilde{r}$.

\item[($p_n$)] For each $0 \leq n \leq m-1$, $p_n = d + (z_n - z_{n-1})$.

\item[($\widetilde{p}_n$)] For each $0 \leq n \leq q-m-1$, $\widetilde{p}_n = \widetilde{d} + (\widetilde{z}_n -\widetilde{z}_{n-1})$.
\end{enumerate}

\begin{rem} \label{rem:pptilde} We can think of each of these sequences as functions of the parameters $q$ and $m$. Then observe that $z(q, \widetilde{m}) = \widetilde{z}(q, m)$, and that $p(q, \widetilde{m}) = \widetilde{p}(q, m)$.
\end{rem}

\begin{rem} \label{rem:ddplus1} It follows from the definition of $(z_n)$ that for all $0 \leq n \leq m-1$, either $z_n = z_{n-1}$, or $z_n = z_{n-1} + 1$, and so by the definition of $(p_n)$, for all $0 \leq n \leq m-1$, $p_n \in \set{d, d+1}$. Similarly, for all $0 \leq n \leq q-m-1$, $\widetilde{p}_n \in \set{\widetilde{d}, \widetilde{d}+1}$. And specifically, $z_{-1} = \widetilde{z}_{-1} = 0$, $z_0 = \widetilde{z}_0 = 1$, $p_0 = d+1$, and $\widetilde{p}_0 = \widetilde{d}+1$.
\end{rem}

If we envision the switch of the curve $g(K_1)$, as in Figure \ref{figure:switchK1}, we can label each of the strands crossing the switch from 0 to $q+m-1$. Beginning at the switch at $0$, we can trace the curve along the blue strand, returning to the switch at $m$, then trace along another blue strand, returning at $2m$, and so on, continuing along blue strands until the first positive integer $k$ for which $km \geq q$, at which point the incoming blue strand will cross-over to red at the switch. Observe that $k = d+1 = p_0$, as $km \geq q$, but $k$ is the least positive integer with this property, and $p_0$ blue arcs have been traversed. Continuing further along the curve traverses a single red strand, returning to the switch at $(d+1)m - q = dm + m - (dm + r) = m - r = z_0m - r$. 

Proceeding along the curve will, again, traverse some number of blue strands, shifting indices at the switch by $m$ each time, until the first time that the index reaches or exceeds $q$. Now, $z_1$ is the smallest integer so that 
\[mz_1 \geq 2r,\]
which is true if and only if 
\[z_0 m - r + m(z_1-z_0) \geq r.\]
Adding $dm$, we obtain
\[ z_0 m - r + (d+ z_1 - z_0)m \geq q,\]
which shows that $p_1 = d+ z_1 - z_0$ is the number of blue strands traversed.

Continuing further along the curve now traverses a single red strand, returning to the switch at $z_0 m - r + (d+ z_1 - z_0)m - q  = z_1m - 2r$.

In general, the $n$th time we repeat this process, we will begin at $z_{n-1}m - nr$, shift the indices by $m$ some number of times until the index reaches or exceeds $q$, and if $z_n$ is the smallest integer so that 
\[m z_n \geq (n+1)r,\]
then
\[z_{n-1}m - nr + (d + z_n - z_{n-1} )m \geq q\]
shows that $p_n = d + z_{n} - z_{n-1}$ is the number of blue strands traversed.

A similar analysis shows that $\widetilde{p}_n$ counts the number of purple strands traversed before a pink strand, now shifting the index by $q-m$ along each purple strand, in the curve $\varepsilon(K_2)$ (see Figure \ref{figure:switchK2}).

\begin{figure}[h]
\begin{center}

\resizebox{\textwidth}{!}{%

\begin{tikzpicture}

\begin{scope}[decoration={
    markings,
    mark=at position 0.4 with {\arrow{>}},
    mark=at position 0.6 with {\arrow{>}}}
    ] 

\foreach \i in {0, 1, 2, 3, 4}
    \draw[ultra thick, red, postaction={decorate}] (\i-8, 1) -- (\i-8, 0);

\end{scope}

\begin{scope}[decoration={
    markings,
    mark=at position 0.5 with {\arrow{>}}}
    ]

\foreach \i in {0, 1, 2, 3, 4}
    \draw[thick, blue, postaction={decorate}] (\i-8, 0) -- (\i-8, -1) node[black, anchor=north] {\i};
\foreach \i in {5, 6, 7, 8, 9, 10, 11}
    \draw[thick, blue, postaction={decorate}] (\i-8, 1) -- (\i-8, 0);
\foreach \i in {5, 6, 7, 8, 9, 10, 11}
    \draw[thick, blue, postaction={decorate}] (\i-8, 0) -- (\i-8, -1) node[black, anchor=north] {\i};
\foreach \i in {12, 13, 14, 15, 16}
    \draw[thick, blue, postaction={decorate}] (\i-8, 1) -- (\i-8, 0);

\end{scope}

\begin{scope}[decoration={
    markings,
    mark=at position 0.4 with {\arrow{>}},
    mark=at position 0.6 with {\arrow{>}}}
    ] 

\foreach \i in {12, 13, 14, 15, 16}
    \draw[ultra thick, red, postaction={decorate}] (\i-8, 0) -- (\i-8, -1) node[black, anchor=north] {\i};

\end{scope}

\draw[ultra thick, black] (-8.5, 0) -- (8.5, 0);

\draw [decorate,
    decoration = {brace}] (-8,1.5) --  (8,1.5);
\node[] at (0, 2) {$q+m$};

\draw [decorate,
    decoration = {brace, mirror}] (-8,-2) --  (-3,-2);
\node[] at (-5.5, -2.5) {$m$};

\draw [decorate,
    decoration = {brace, mirror}] (-3,-2) --  (2,-2);
\node[] at (-0.5, -2.5) {$m$};

\draw [decorate,
    decoration = {brace, mirror}] (2,-2) --  (7,-2);
\node[] at (4.5, -2.5) {$m$};

\draw [decorate,
    decoration = {brace, mirror}] (-6,-3) --  (5,-3);
\node[] at (-0.5, -3.5) {$d+1$};

\end{tikzpicture}
}
\end{center}

\caption{An example of a switch for $\alpha\beta\gamma \delta^{-1} \gamma \delta^{2}(K_1)$ with $q = 12$, and $m = 5$.}
\label{figure:switchK1}
\end{figure}
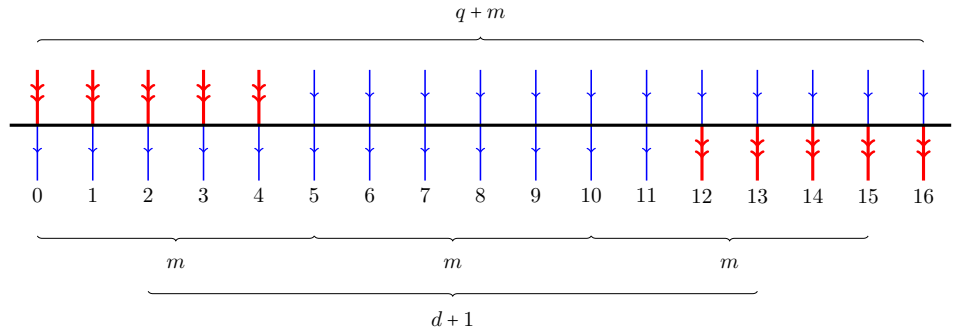

\begin{figure}[h]
\begin{center}

\resizebox{\textwidth}{!}{%

\begin{tikzpicture}

\begin{scope}[decoration={
    markings,
    mark=at position 0.4 with {\arrow{>}},
    mark=at position 0.6 with {\arrow{>}}}
    ]

\foreach \i in {0, 1, 2, 3, 4, 5, 6}
    \draw[ultra thick, magenta, postaction={decorate}] (\i-8, 1) -- (\i-8, 0);

\end{scope}

\begin{scope}[decoration={
    markings,
    mark=at position 0.5 with {\arrow{>}}}
    ] 
    
\foreach \i in {0, 1, 2, 3, 4, 5, 6}
    \draw[thick, violet, postaction={decorate}] (\i-8, 0) -- (\i-8, -1) node[black, anchor=north] {\i};
\foreach \i in {7, 8, 9, 10, 11}
    \draw[thick, violet, postaction={decorate}] (\i-8, 1) -- (\i-8, 0);
\foreach \i in {7, 8, 9, 10, 11}
    \draw[thick, violet, postaction={decorate}] (\i-8, 0) -- (\i-8, -1) node[black, anchor=north] {\i};
\foreach \i in {12, 13, 14, 15, 16, 17, 18}
    \draw[thick, violet, postaction={decorate}] (\i-8, 1) -- (\i-8, 0);
    
\end{scope}

\begin{scope}[decoration={
    markings,
    mark=at position 0.4 with {\arrow{>}},
    mark=at position 0.6 with {\arrow{>}}}
    ] 
    
\foreach \i in {12, 13, 14, 15, 16, 17, 18}
    \draw[ultra thick, magenta, postaction={decorate}] (\i-8, 0) -- (\i-8, -1) node[black, anchor=north] {\i};

\end{scope}

\draw[ultra thick, black] (-8.5, 0) -- (10.5, 0);

\draw [decorate,
    decoration = {brace}] (-8,1.5) --  (10,1.5);
\node[] at (0, 2) {$2q-m$};

\draw [decorate,
    decoration = {brace, mirror}] (-8,-2) --  (-1,-2);
\node[] at (-4.5, -2.5) {$q-m$};

\draw [decorate,
    decoration = {brace, mirror}] (-1,-2) --  (6,-2);
\node[] at (2.5, -2.5) {$q-m$};

\draw [decorate,
    decoration = {brace, mirror}] (-5,-3) --  (3,-3);
\node[] at (-1, -3.5) {$\widetilde{d}+1$};

\end{tikzpicture}
}
\end{center}

\caption{An example of a switch for $\varepsilon(K_2)$ with $q = 12$, and $q-m = 7$.}
\label{figure:switchK2}
\end{figure}
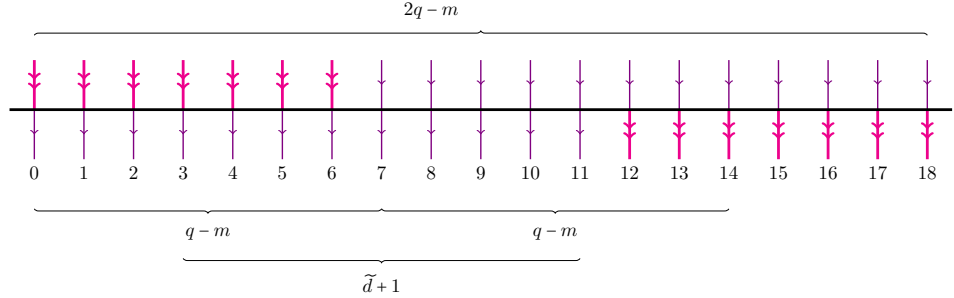

\begin{lem} \label{lem:nearpall} $p_0 = d+1$, $p_{m-1} = d$, and for all $1 \leq n \leq m-2$, 
\[p_n = p_{m-1 - n}.\]
\end{lem} 

\begin{proof}
That $p_0 = d+1$ is explained in Remark \ref{rem:ddplus1}. That $p_{m-1} = d$ can be seen from the discussion above that $p_n$ counts the number of times blue strands are traversed before crossing over to a red strand. When this process of traversing the stands ends, and the curve closes up, it is because the sequence of blue strands terminates on $q$, after which the left-most red strand on the bottom connects to the left-most red strand on top. This must mean that the initial strand after the $(m-1)$st cycle was $r$, so that the last cycle adds exactly $d$ copies of $m$ to end on $r + dm = q$.

Now, the rest of the equations follow from a symmetry in traversing the curve forwards or backwards. On the one hand, we can follow the curve for one full cycle, traversing $p_0 = d+1$ blue stands, then one red strand, ending at index $m-r$. On the other hand, if we had followed the curve with the reverse orientation, traversing first the red strand connecting index $0$ to index $q$, then $p_{m-1} = d$ blue strands, and then another red strand, we end at index $q + r$. Now, $m-r + pm \geq q$ if and only if $q + r -pm \leq m$, which shows that $p_1 = p_{m-2}$. In fact, as $m-r$ and $q + r$ are symmetric about $\frac{q + m}{2}$, any amount of cycling forward from $m-r$ by $m$, mod $q+m$, will be mirrored by cycling backward from $q+r$ by $m$, mod $q+m$, so that $(m-r + pm - nq ) \geq q$ if and only if $q + r - pm + nq \leq m$, and $p_n$ will always match $p_{m - 1 - n}$ for $1 \leq n \leq m-2$.
\end{proof}

\begin{lem}
\label{lem:pptilde}
For $m < \frac{q}{2}$, we have the following:
\begin{enumerate}
\item $\widetilde{p}_n = 2$ precisely when $n = 0$ or  \[n = (p_0 + p_1 + \cdots + p_j - (j+2))\] for $j = 0, 1, \dots, (m-2)$.
\item If $\left( \widetilde{p}_{n_i} \right)$ for $0 \leq i \leq m-1$ is the subsequence of $2$'s,
then \[p_0 = 1 + (q-m) - n_{m -1},\]
and for all $0 \leq i \leq m-2$, 
\[p_{m -1- i} = 1 + (n_{i+1} - n_{i}).\] In other words, the sequence $\left(p_n \right)$ is the reverse of the sequence whose terms are the sum of each 2 in $\left(\widetilde{p}_n \right)$ with all subsequent 1's before the next 2.
\end{enumerate}

\end{lem}

\begin{proof}
Observe that $\widetilde{p}_0 = 2$. Then, notice that by cycling through strands, for each $j$,
\[(p_0 + p_1 + \cdots + p_j - 1)m < (j+1)q \leq (p_0 + p_1 + \cdots + p_j)m,\]
which is equivalent to
\begin{equation} \tag{$j$} (p_0 + p_1 + \cdots + p_j - (j+2))m < (j+1)(q-m) \leq (p_0 + p_1 + \cdots + p_j - (j+1))m. \end{equation}

Since $d > 1$, each $p_j \geq 2$, so these sums, $(p_0 + p_1 + \cdots + p_j - (j+2))$, will be strictly increasing with $j$. The inequality $(j)$ says precisely that $(j+1)$ multiples of $(q-m)$ are greater than the $(p_0 + p_1 + \cdots + p_j - (j+3) + 1)$th multiple of $m$, but not greater than the next. This means that $\widetilde{z}_{(p_0 + p_1 + \cdots + p_j - (j+3))} = j+1$, while $\widetilde{z}_{(p_0 + p_1 + \cdots + p_j - (j+2))} = j+2$, so that $\widetilde{p}_{(p_0 + p_1 + \cdots + p_j - (j+2))} = 1 + (j+2) - (j+1) = 2$. 

On the other hand, say $n$ is between $n_j = (p_0 + p_1 + \cdots + p_j - (j+2))$ and $n_{j+1} = (p_0 + p_1 + \cdots + p_j + p_{j+1} - (j+3))$. In this case, we have $\widetilde{z}_{n_j} = j+2$, $\widetilde{z}_{n_{j+1} - 1} = j+2$, and $\widetilde{z}_{n_{j+1}} = j+3$. By definition of $\widetilde{z}$, this says that
\[ ((n_{j+1}-1) + 1)m \leq (q-m)(j+2) < (n_{j+1} + 1)m \leq (q-m)(j+3).\]
But then $n \leq n_{j+1} -1 < n_{j+1}$, which means that
\[ (n+1)m \leq ((n_{j+1}-1) + 1)m \leq (q-m)(j+2),\]
which establishes that $\widetilde{z}_{n} \leq j+2$. As $\widetilde{z}$ is a non-decreasing sequence, and $n_j \leq n-1 < n$ this shows that $\widetilde{p}_n = 1 + (\widetilde{z}_n - \widetilde{z}_{n-1}) = 1 + ((j+2) - (j+2)) = 1 \neq 2$.

To prove the second statement, observe that by the statement above, for each $1 \leq i \leq m-2$, 
\begin{align*} 1+ n_{i+1} - n_i && = & 1 + (p_0 + p_1 + \cdots + p_i - (i+2)) - (p_0 + p_1 + \cdots + p_{i-1} - (i+1)) \\
&& = &\qquad p_{i}.
\end{align*}
By Lemma \ref{lem:nearpall}, this is equal to $p_{m-1-i}$.

For $i = 0$, we get
\[1 + n_1 - n_0 = 1 + (p_0 - 2) - 0 = p_0 - 1 = p_m,\]
by Lemma \ref{lem:nearpall} again.

And finally, 
\begin{align*} 1+ (q-m) - n_{m-1} & = & 1 + (q-m) - (p_0 + p_1 + \cdots + p_{m-2} - (m-2+2)) \\
& = & 1 + (q-m) - (p_0 + p_1 + \cdots + p_{m-2} + p_{m-1} - m) + p_{m-1}\\
& = & 1 + (q-m) - (q - m) + p_{m-1}\\
& = & 1 + p_{m-1}\\
& = & p_0.
\end{align*} 
\end{proof}

\begin{rem} \label{rem:exchange} A similar statement then holds if $m > \frac{q}{m}$, exchanging the roles of $p$ and $\widetilde{p}$. \end{rem}

We now begin showing that $g(K_1)$ and $\varepsilon(K_2)$ represent the same free homotopy elements in both handlebodies, $H'$ and $H$, bounded by $F$. For a curve $K$ on $F$, let $[K]_{H'}$ and $[K]_{H}$ represent the free homotopy classes of $K$ in the fundamental groups $\pi_1(H')$ and $\pi_1(H)$, respectively.

For the inside handlebody, $H$, using the standard cut system defined by the curves $\textbf{a}$ and $\textbf{x}$, we get a word in $\pi_1(H)$. For convenience, we will call $A$ and $X$ the corresponding generators of $\pi_1(H)$, we will denote inverses by $\overline{A}$ and $\overline{X}$, respectively, so that each time a curve crosses $\textbf{a}$ with a positive orientation, we contribute an $A$ to the word representing the curve.

For the outside handlebody $H'$, we will similarly use the curves $\textbf{b}$ and $\textbf{y}$, and call $B$ and $Y$ the corresponding generators for $\pi_1(H')$.

Considering first the inside handlebody $H$, each blue strand contributes an $\overline{A}X$, and each red strand contributes an $A$, so that the word $[g(K_1)]_{H}$ will be
\begin{eqnarray*} &  A\left(\overline{A}X\right)^{p_{m-1}} A\left(\overline{A}X\right)^{p_{m-2}} \cdots A\left(\overline{A}X\right)^{p_{1}} A\left(\overline{A}X\right)^{p_{0}} \\
 = & X\left(\overline{A}X\right)^{p_{m-1}-1} X\left(\overline{A}X\right)^{p_{m-2}-1} \cdots X\left(\overline{A}X\right)^{p_{1}-1} X\left(\overline{A}X\right)^{p_{0}-1}. \end{eqnarray*}

Observe that regardless of the values of $p(q, m)$, this will always be a word with a single $X$ on the left, followed by an alternating pattern of $\overline{A}$ and some powers of $X$, ending on $(\overline{A}X)^{p_{0}-1}$ on the right. We will eventually simplify and standardize this word by cycling the right-most power of $X$ to the left, and record the powers of $X$ as a sequence $\left(\chi^1_n\right)$.

\begin{lem} \label{lem:chiptilde} $\chi^1_n = \widetilde{p}_n$ for all $0 \leq n \leq q-m-1$.
\end{lem}

\begin{proof}

We will treat two cases separately. They are in some very strong sense dual cases, but the analysis required is slightly different in each case.

\textbf{Case 1: $d > 1$}

In this case, all of the $p_n \geq 2$, so the powers appearing in the word $[g(K_1)]_{H}$ are all at least 1, and $\chi^1_n \in \set{1, 2}$ for all $n$, and the word becomes,
\begin{small}
\begin{gather*} X\left(\overline{A}X\right)^{p_{m-1}-1} X\left(\overline{A}X\right)^{p_{m-2}-1} \cdots X\left(\overline{A}X\right)^{p_{1}-1} X\left(\overline{A}X\right)^{p_{0}-1} \\
=  X\left(\overline{A}X\right)^{p_{m-1}-2} \left(\overline{A}X\right) X\left(\overline{A}X\right)^{p_{m-2}-2} \left(\overline{A}X\right) \cdots X\left(\overline{A}X\right)^{p_{1}-2} \left(\overline{A}X\right) X\left(\overline{A}X\right)^{p_{0}-2} \left(\overline{A}X\right) \\
=  X\left(\overline{A}X\right)^{p_{m-1}-2} \left(\overline{A}X^2\right) \left(\overline{A}X\right)^{p_{m-2}-2} \left(\overline{A}X^2\right) \cdots \left(\overline{A}X\right)^{p_{1}-2} \left(\overline{A}X^2\right) \left(\overline{A}X\right)^{p_{0}-2} \left(\overline{A}X\right). \end{gather*} 
\end{small}

Analyzing the word after the right-most $X$ is cycled to the left, we can see that $\chi^1_n = 2$ exactly when $n$ is equal to
\begin{flalign*}
& 0,\\
& p_{m-1}    -1, \\
& p_{m-1} + p_{m-2}   - 2, \\
& p_{m-1} + p_{m-2} + p_{m-3}   - 3,\\
& \vdots\\
& p_{m-1} + p_{m-2} + \cdots + p_2 - (m-2), \mbox{ or}\\
& p_{m-1} + p_{m-2} + \cdots + p_3 + p_1 - (m-1).\\
\end{flalign*}

In light of Lemma \ref{lem:nearpall}, this is the same as $n$ equal to
\begin{align*} 
& 0, \\
& p_0 - 2,\\
& p_0 + p_1 - 3,\\
& p_0 + p_1 + p_2 - 4,\\
& \vdots \\
& p_0 + p_1 + \cdots p_j - (j+2),\\ 
& \vdots\\
& p_0 + p_1 + \cdots + p_{m-3} - (m-1), \mbox{ or}\\
& p_0 + p_1 + \cdots + p_{m-2} + p_{m-2} - m.\\
\end{align*}

This is precisely when $\widetilde{p}_n = 2$ according to Lemma \ref{lem:pptilde}.

\textbf{Case 2: $d=1$}

Now $p_n \in \set{1, 2}$ for all $0 \leq n \leq m-1$. Suppose $\left( p_{n_i} \right)$ for $0 \leq i \leq q-m-1$ is the subsequence of $2$'s. We note that $X\left(\overline{A}X\right)^{0} = X$, while $X\left(\overline{A}X\right)^{1} = X\overline{A}X$, so for all $1 \leq i \leq q-m-1$, 
\[\chi^1_0 = 1 + m - n_{q-m -1},\]
and 
\[\chi^1_{q - m - i} = 1 + (n_{i} - n_{i-1}).\]

In other words, the sequence $\left(\chi^1_n\right)$ is the reverse of the sequence whose terms are the sum of each 2 in $\left(p_n \right)$ with all subsequent 1's before the next 2.

This is precisely $\widetilde{p}_n$ by the second statement from, and Remark \ref{rem:exchange} following, Lemma \ref{lem:pptilde}.
\end{proof}

Turning now to the other curve, we have a much more straightforward relationship between the word and the sequence $\widetilde{p}_n(q, m)$. Each purple strand contributes an $X$, and each pink strand contributes an $\overline{A}$, so that the word $[\varepsilon(K_2)]_{H}$ will be
\[ \overline{A}X^{\widetilde{p}_{q-m-1}} \overline{A}X^{\widetilde{p}_{q-m-2}} \cdots \overline{A}X^{\widetilde{p}_{1}} \overline{A}X^{\widetilde{p}_{0}}. \]

This will always be a word with a single $\overline{A}$ on the left, followed by an alternating pattern of some powers of $X$ and $\overline{A}$. Again, we will simplify and standardize this word, this time by cycling the left-most $\overline{A}$ to the right, so as to match the form of $[g(K_1)]_{H}$, and record the powers of $X$ as a sequence $\left(\chi^2_n\right)$. It is apparent that $\left( \chi^2_n \right)$ is precisely the reverse of $\left( \widetilde{p}_n \right)$.

Then Lemma \ref{lem:chiptilde} shows that the words $[g(K_1)]_{H}$ and $[\varepsilon(K_2)]_{H}$ are mirror images of each other, so by Lemma \ref{lem:cyclicpallindrome}, $g(K_1)$ and $\varepsilon(K_2)$ represent the same free homotopy class in $\pi_1(H)$.

Next, we consider the outside handlebody $H'$. Here, each blue strand of $g(K_1)$ contributes an $BY^2$, and each red strand contributes a $\overline{Y}\overline{B}$, so that the word $[g(K_1)]_{H'}$ will be
\begin{eqnarray*} &  \overline{Y}\overline{B}(BY^2)^{p_{m-1}} \overline{Y}\overline{B}(BY^2)^{p_{m-2}} \cdots \overline{Y}\overline{B}(BY^2)^{p_{1}} \overline{Y}\overline{B}(BY^2)^{p_{0}} \\
 = & Y(BY^2)^{p_{m-1}-1} Y(BY^2)^{p_{m-2}-1} \cdots Y(BY^2)^{p_{1}-1} Y(BY^2)^{p_{0}-1}. \end{eqnarray*}

Observe that regardless of the values of $(p_n)$, this will always be a word with a power of $Y$ on the left, followed by an alternating pattern of $B$ and some powers of $Y$, ending on $Y^2$ on the right. We will simplify and standardize this word by cycling the right-most $Y^2$ to the left, and record the powers of $Y$ as a sequence $\left(\Upsilon^1_n\right)$.

\begin{lem} \label{lem:upsilonptilde} $\Upsilon^1_n = \widetilde{p}_n + 1$ for all $0 \leq n \leq q-m-1$.
\end{lem}

\begin{proof}
The word $[\varepsilon(K_2)]_{H'}$ has exactly the same structure as $[g(K_1)]_{H}$, with $Y^3$ in place of $X^2$, and $B$ in place of $\overline{A}$, so $\Upsilon^1_n = \widetilde{p}_n + 1$ for all $0 \leq n \leq q-m-1$.
\end{proof}

And again, the word representing the second curve, $\varepsilon(K_2)$, has a much more straightforward relationship with the sequence $\left( \widetilde{p}_n \right)$. Each purple strand contributes a $Y$, and each pink strand contributes a $YB$, so that the word $[\varepsilon(K_2)]_{H'}$ will be
\begin{eqnarray*} & YBY^{\widetilde{p}_{q-m-1}} YBY^{\widetilde{p}_{q-m-2}} \cdots YBY^{\widetilde{p}_{1}} YBY^{\widetilde{p}_{0}} \\
= & YBY^{\widetilde{p}_{q-m-1}+1} BY^{\widetilde{p}_{q-m-2}+1} \cdots BY^{\widetilde{p}_{1}+1} BY^{\widetilde{p}_{0}}.
\end{eqnarray*}

Cycling the left-most $YB$ to the right, we obtain an alternating pattern of $B$ and powers of $Y$, starting with a power of $Y$ on the left and ending in a $B$, matching the form of $[g(K_1)]_{H'}$, and we'll record the powers of $Y$ as $\left(\Upsilon^2_n\right)$. It is clear that $\left( \Upsilon^2_n \right)$ is the reverse of the sequence $\left( \widetilde{p}_n + 1\right)$. Then, by Lemmas \ref{lem:upsilonptilde} and \ref{lem:cyclicpallindrome}, $g(K_1)$ and $\varepsilon(K_2)$ represent the same free homotopy class in $\pi_1(H')$.

\subsubsection{The homotopy obstruction is conclusive}
\label{subsub:HomotopyConclusive}

It is worth noting that Dehn twisting $g(K_1)$, say, around any separating curve that is \emph{not} a reducing sphere curve will provide examples of curves that are still homologous, but are not Goeritz equivalent due to Theorem \ref{thm:Homotopy}'s obstruction.
  
It would be ideal if the homotopy conditions of Theorem \ref{thm:Homotopy} were sufficient to conclude Goeritz equivalence as well.

\begin{question} Suppose $K_1$ and $K_2$ are homologous simple closed curves with non-zero surface slopes. Are $K_1$ and $K_2$ Goeritz equivalent if and only if $K_1$ and $K_2$ are freely homotopic in $H$, and freely homotopic in $H'$?
\end{question}

%
%
%
\bibliographystyle{gtart}
\bibliography{Genus2GoeritzEquivalence}
%

%

\end{document}